\documentclass[11pt]{amsart}  
\usepackage{amssymb,amsmath,amsthm,amscd}
\usepackage{graphicx}
\usepackage{tikz-cd}
\usepackage{lscape}
\newcommand{\ep}[0]{\epsilon}

\newcommand{\der}[2]{\frac{d{#1}}{d{#2}}}

\newcommand{\ov}[1]{\frac{1}{#1}}
\newcommand{\f}[1]{\mathbb{#1}}

\newcommand{\ar}[1]{\mathbb{R}^{#1}}
\newcommand{\rp}[1]{\ar{}\text{P}^{#1}}

\newcommand{\mods}[1]{\mathfrak{M}_{\text{#1}>0}}
\newcommand{\modse}[1]{\mathfrak{M}_{\text{#1}\geq0}}

\newcommand{\zco}[2]{H^{#2}(#1,\f{Z})}
\newcommand{\twoco}[2]{H^{#2}(#1,\f{Z}_2)}
\newcommand{\cpcs}[2]{\#^{#1}\f{C}P^2\#^{#2}\overline{\f{C}P^2}}

\numberwithin{equation}{section}

\newtheoremstyle{fancy1}{10pt}{10pt}{\itshape}{12pt}{\textsc\bgroup}{.\egroup}{8pt}{
}
\newtheoremstyle{fancy2}{10pt}{10pt}{}{12pt}{\itshape}{.}{8pt}{ }

\theoremstyle{fancy1}

\newtheorem{cor}[equation]{Corollary}
\newtheorem{lem}[equation]{Lemma}

\newtheorem{thm}[equation]{Theorem}
\newtheorem*{thm*}{Theorem}

\newtheorem{main}{Theorem}
\newtheorem*{main*}{Theorem}

\newtheorem*{cor*}{Corollary}
\newtheorem*{prop*}{Proposition}

\newtheorem*{problem*}{Problem}

\theoremstyle{fancy2}

\newtheorem{rem}[equation]{Remark}
\newtheorem*{rems*}{Remarks}

\newtheorem*{rem*}{Remark}

\newtheorem*{example*}{Example}

\newcommand{\cref}[1]{Corollary~\ref{#1}}

\newcommand{\lref}[1]{Lemma~\ref{#1}}

\newcommand{\tref}[1]{Theorem~\ref{#1}}
\newcommand{\sref}[1]{Section~\ref{#1}}

\renewcommand{\H}{\ensuremath{\operatorname{H}}}

\newcommand{\F}{\ensuremath{\operatorname{F}}}

\newcommand{\SO}{\ensuremath{\operatorname{SO}}}

\renewcommand{\O}{\ensuremath{\operatorname{O}}}

\newcommand{\Spin}{\ensuremath{\operatorname{Spin}}}
\newcommand{\Pin}{\ensuremath{\operatorname{Pin}}}

\def\con#1=#2(#3){#1 \equiv #2 \bmod{#3}}

\begin{document}

\title{Moduli spaces of Ricci positive metrics in dimension five}

\author{McFeely Jackson Goodman}
\address{University of California, Berkeley}
\email{mjgoodman@berkeley.edu}
\thanks{This research was partially supported by National Science Foundation grant  DMS-2001985}

\begin{abstract}
We use the \(\eta\) invariants of spin\(^c\) Dirac operators to distinguish connected components of moduli spaces of  Riemannian metrics with positive Ricci curvature.  We then find infinitely many non-diffeomorphic five dimensional manifolds for which these moduli spaces each  have infinitely many components.  The manifolds are total spaces of principal \(S^1\) bundles over \(\#^a\f{C}P^2\#^b\overline{\f{C}P^2}\) and the metrics are lifted from Ricci positive metrics on the bases. Along the way we classify   5-manifolds with fundamental group \(\f{Z}_2\) admitting free \(S^1\) actions with simply connected quotients.    
\end{abstract}

\maketitle

Many closed manifolds are known to admit Riemannian metrics of positive Ricci curvature, for example, all compact, simply connected homogeneous spaces, biquotients, and cohomogeneity one manifolds, see \cite{Be}, \cite{GZric}, \cite{ST}.  Systematic methods for constructing such metrics on certain connected sums and bundles have been explored in \cite{CG}, \cite{GPT}, \cite{N}, \cite{SW}, \cite{SY1}, \cite{Wcs}.

Once we know that a manifold admits positive Ricci curvature we ask how many such metrics it admits.  The space of geometrically distinct metrics of positive Ricci curvature on a manifold \(M\) is the moduli space \(\mods{Ric}(M)=\mathfrak{R}_{\text{Ric}>0}(M)/\text{Diff}(M),\) where \(\mathfrak{R}_{\text{Ric}>0}(M)\) is the set of positive Ricci curvature metrics on \(M\) and  Diff\((M)\) is the diffomorphism group, acting by pullbacks.  The number of path components of \(\mods{Ric}\) serves as a coarse measure of distinct positive Ricci curvature metrics on \(M\).  

We  identify an infinite family of 5-manifolds \(M\) with \(\pi_1(M)=\f{Z}_2\) such that \(\mods{Ric}(M)\) has infinitely many path components. 

\begin{main}\label{thm}
	Let \(B^4=\#^a\f{C}P^2\#^b\overline{\f{C}P^2}\), \(a+b\geq2\), and let \(S^1\to M^5\to B^4\) be a principal bundle with first Chern class 2d, where \(d\in H^2(B^4,\f{Z})\) is  primitive and \(w_2(TB^4)=d\) mod 2.  Then \(\mods{Ric}(M^5)\) has infinitely many path components.  
	
\end{main}

\noindent Here \(w_2\) is the second Stiefel-Whitney class and a primitive class is one that is not a positive integer multiple of any other. We will see that for each 4-manifold \(B\) there are 2, 3 or 4 diffeomorphism types of such total spaces \(M\), depending on the value of \(|a-b|\) mod 4, each of which admits infinitely many inequivalent free \(S^1\) actions with quotient \(B.\)
  The only other five dimensional manifolds for which \(\mods{Ric}\)  is known to have infinitely many components are the four homotopy real projective spaces recently described by Dessai and Gonz\'{a}lez-\'{A}lvaro \cite{DG} and five quotients of \(S^2\times S^3\) recently described by  Wermelinger \cite{We}.  

The conditions on the first Chern class in \tref{thm} are equivalent to the statement that \(\pi_1(M^5)=\f{Z}_2\), \(M^5\) is non-spin, and the universal cover of \(M^5\) is spin.  \(M^5\) can be constructed by taking five dimensional homotopy real projective spaces, removing  tubular neighborhoods of generators of the fundamental group, and gluing  along the boundaries of the tubular neighborhoods.  By the classification of Smale \cite{Sm} and  Barden \cite{B}, the universal cover \(\tilde M^5\) is diffeomorphic to \(\#^{a+b-1}S^3\times S^2.\) But we do not know an explicit description of the deck group action by \(\f{Z}_2\) on \(\tilde{M^5}\). We construct the metrics in \tref{thm} by lifting metrics from quotients of \(M^5\) by free \(S^1\) actions.

Our second theorem identifies   conditions under which \(M^5\) admits one, and infinitely many, free \(S^1\) actions.   Here \(b_2(M)\) is the second Betti number of \(M\).

\begin{main}\label{topthm}
	Let \(M^5\) be a 5-manifold with \(\pi_1=\f{Z}_2.\) Then \(M\) admits a free \(S^1\) action with a simply connected quotient if and only if \(M\) is orientable, \(H_2(M,\f{Z})\) is torsion free and \(\pi_1(M)\) acts trivially on \(\pi_2(M).\)  Furthermore if \(b_2(M)=0\) then \(M\) is diffeomorphic to \(\ar{}P^5\).  If \(b_2(M)>0\) and \(M\) admits a free \(S^1\) action with simply connected quotient \(B^4\) then \(M\) admits infinitely many inequivalent free \(S^1\) actions with quotients diffeomorphic to \(B^4\).    
	
	\end{main}

    Note that here \(B^4\) can be any simply connected 4-manifold, and need not be one of the manifolds of \tref{thm}.  \tref{bases} provides greater detail about the correspondence between a 5 manifold \(M^5\) and the set \(Q(M)\) of possible quotients \(B^4=M^5/S^1\).  Given  \(M^5\) satisfying the hypotheses of \tref{topthm}, we give conditions on the cohomology ring of a 4-manifold \(B^4\) which are necessary and sufficient for \(B\) to be in \(Q(M).\)  In particular, any smooth manifold homeomorphic to a manifold in \(Q(M)\) is in \(Q(M).\)  In \cref{goodbases} we see that for any such \(M\),  \(Q(M)\) contains either \(\#^cS^2\times S^2\) or \(\cpcs{a}{b}\) for some \(a,b,c\in\f{Z}.\) Those manifolds admit metrics with positive Ricci curvature, which can be lifted to \(M\).  Thus we have:
    
    \begin{cor*}
    		Let \(M\) be a 5 manifold with \(\pi_1(M)=\f{Z}_2\) admitting a free \(S^1\) action with a simply connected quotient.  Then \(M\) admits a metric with positive Ricci curvature. 
    \end{cor*}
    
    \noindent Furthermore, it follows from \tref{bases} that given a simply connected 4-manifold \(B^4,\) the set of diffeomorphism types of total spaces \(M^5\) with \(\pi_1(M^5)=\f{Z}_2\) of \(S^1\) bundles over \(B^4\) depends only on the cohomology ring of \(B^4\).  In particular \tref{thm} would describe the same set of 5-manifolds if we replaced \(\cpcs{a}{b}\) with one of the manifolds homeomorphic to it.



\smallskip

We first review previous work with methods and results relevant to \tref{thm}.  In \cite{KS} Kreck and Stolz invented a moduli space invariant \(s(M,g)\in\f{Q}\) for a metric \(g\) of positive scalar curvature on a closed spin manifold \(M\).  The metric is based on the \(\eta\) spectral invariant of the Dirac operator defined in \cite{APS}.  If \(s(M,g_1)\neq s(M,g_2)\) then \(g_1\) and \(g_2\) represent elements in different path components of \(\mods{scal}\).  The authors use the invariant to prove that for \(M^{4k+3}\) with a unique spin structure and vanishing rational Pontryagin classes \(\mods{scal}(M)\) is either empty or has infinitely many components.  

Since a path of Riemannian metrics which maintains positive Ricci curvature  maintains positive scalar curvature as well, the \(s\) invariant can detect connected components of \(\mods{Ric}\).  Kreck and Stolz calculated \(s\) for the Einstein metrics on \(S^1\) bundles \(N^7_{k,l}\) over \(\f{C}P^1\times \f{C}P^2\) described by Wang and Ziller \cite{WZ}.  Using the diffeomorphism classification in \cite{KS0}, they showed that when \(k\) is even and gcd\((k,l)=1\), \(N_{k,l}\) is diffeomorphic to infinitely many manifolds in the same family.  As the \(s\) invariant takes infinitely many values on those metrics, the authors concluded that \(\mods{Ric}(N_{k,l})\) has infinitely many components.  Similar results have since been proved for \(S^1\) bundles over \(\f{C}P^{1}\times\f{C}P^{2n}\), \(n\geq1\), see \cite{DKT}.   

Wraith showed that for a homotopy sphere \(\sigma^{4k-1}\) bounding a parallelisable manifold,  \(\mods{Ric}(\sigma)\) has infinitely many components.  The procedure known as plumbing with disc bundles over spheres produces infinitely many parallelisable manifolds with boundaries diffeomorphic to \(\sigma\).  Wraith constructed metrics of positive Ricci curvature on each boundary in \cite{Wes} and calculated the \(s\) invariant of each metric in  in \cite{Wmo}.  

Dessai \cite{D} and the author \cite{G} used the \(s\) invariant to find several infinite families of 7-dimensional sphere bundles \(M^7\) such that \(\mods{Ric}(M)\) and \(\modse{sec}(M)\) have infinitely many path components.  Grove and Ziller \cite{GZms, GZlift}  constructed metrics of nonnegative sectional curvature on the manifolds in those families, and the diffeomorphism classifications in \cite{CE} and \cite{EZ} show that each manifold is diffeomorphic to infinitely many other members of the family.       

More recently, Dessai and Gonz\'{a}lez-\'{A}lvaro \cite{DG} showed that if \(M^5\) is one of the four closed manifolds homotopy equivalent to \(\ar{}P^5\) then \(\modse{sec}(M)\) and \(\mods{Ric}(M)\) have infinitely many path components.    L\'{o}pez de Medrano \cite{L} showed that each such \(M^5\) admits infinitely many descriptions as a quotient of a Brieskorn variety, and Grove and Ziller showed the each quotient admits a metric of nonnegative sectional curvature \cite{GZric}. Dessai and  Gonz\'{a}lez-\'{A}lvaro calculated the relative \(\eta\) invariant for those metrics to distinguish the path components.  Wermelinger extended their method to prove the same conclusion for five \(\f{Z}_2\) quotients of \(S^2\times S^3\) in \cite{We}.


We now outline the proof of \tref{thm}.  We use \tref{topthm} to show that each manifold \(M^5\) in \tref{thm} admits infinitely many inequivalent free \(S^1\) actions with quotient \(B^4=\#^a\f{C}P^2\#^b\overline{\f{C}P^2}.\) We modify a result of Perelman \cite{P} to show that \(B\) admits a metric of positive Ricci curvature.  That metric can be lifted to a metric of positive Ricci curvature on \(M\) by \cite{GPT}.  The lifted metrics depend on the \(S^1\) action, and  we get infinitely many distinct metrics on \(M\). 

 We  show that in dimensions \(4k+1,\) the \(\eta\) invariant of a certain spin\(^c\) Dirac operator constructed for a positive Ricci curvature metric \(g\) depends only on the connected component of the class of \(g\) in \(\mods{Ric}\).    To complete the proof we calculate \(\eta\) for each metric on \(M\) and show that it obtains infinitely many values.  This is the most intricate part of our proof.     
   
   The standard method for calculating the \(\eta\) invariant of a spin Dirac operator on a manifold \(M\) with positive scalar curvature is to extend the metric over a manifold \(W\) with \(\partial W=M\) such that the extension has positive scalar curvature as well.  When \(M\) is not spin but spin\(^c,\)  both the metric and a  unitary connection on the complex line bundle associated to the spin\(^c\) structure must be extended.  The desired condition then involves the curvatures of both metric and connection.  In their work, Dessai and Gonz\'{a}lez-\'{A}lvaro passed to the universal cover to find a suitable \(W\) over which the connection could be extended to a flat connection.  They use equivariant \(\eta\) invariants on the cover to compute the \(\eta\) invariant on the quotient.  
   
   In this paper we work directly on \(M\) and use  a manifold with boundary \(W\) over which the connection cannot be extended to a flat connection, but the curvature of the extension can be explicitly controlled.  To be specific, we extend the metric and connection on \(M\) to a metric \(h\) and connection \(\nabla\) on the disc bundle \(W=M\times_{S^1}D^2\) associated to the \(S^1\) bundle.  We then use the Atiyah-Patodi-Singer index theorem \cite{APS} to obtain a formula for  \(\eta\) in terms of the index of the spin\(^c\) Dirac operator on \(W\) and topological data on \(W\).  The index will vanish as long as
   \[\text{scal}(h)>2|F^\nabla|_h\]
   where \(F^\nabla\) is the curvature form of the connection \(\nabla.\)  We accomplish the extension for a general class of \(S^1\) invariant metrics of positive scalar curvature.  This is more general than we need but may be of independent interest. In fact we construct \(h\) and \(\nabla\) such that 
    \[\text{scal}(h)>\ell|F^\nabla|_h\]
   where \(\ell\) is a positive integer such that the first Chern class of the \(S^1\) bundle is \(\ell\) times the canonical class of a spin\(^c\) structure on the quotient.

  Sha and Yang  constructed metrics of positive Ricci curvature on the 4-manifolds \(\#^{a-b}\f{C}P^2\#^b S^2\times S^2\), \(a>b\), in \cite{SY}.  Those manifolds are diffeomorphic to \(\#^{a}\f{C}P^2\#^{b}\overline{\f{C}P^2}\), and so a manifold \(M\) satisfying the hypotheses of \tref{thm} also admits a free \(S^1\) action with quotient \(\#^{a-b}\f{C}P^2\#^bS^2\times S^2\).  One can lift the Sha-Yang metric to \(M\), and there is no reason to expect that the resulting metric lies in the same component as the metric lifted from \(\cpcs{a}{b}\) in the proof of \tref{thm}.  We will see, however, that the computation of the \(\eta\) invariant involves only the cohomology ring of the quotient, and we cannot distinguish any new components in this way.
  
  In \cite{SY1} Sha and Yang also found metrics of positive Ricci curvature on \(\#^b S^2\times S^2.\)  One might expect our methods to yield a similar result in this case.  The 5-manifolds, however, would be spin, and the eta invariant of the spin Dirac operator in dimension 4k+1 vanishes, even when twisted with certain complex line bundles, see \cite{BG1}.    

We now discuss  \tref{topthm}.  In \cite{HS}, Hambleton and Su find a complete diffeomorphism classification of 5-manifolds \(M\) with \(\pi_1(M)=\f{Z}_2\) when \(M\) is orientable, \(H_2(M,\f{Z})\) is torsion free, and \(\pi_1(M)\) acts trivially on \(\pi_2(M).\)  They apply the classification to investigate the diffeomorphism type of the total space of an \(S^1\) bundle over a simply connected 4-manifold.  When the total space is non-spin but has a spin universal cover, as is the case in \tref{thm}, they can only restrict the diffeomorphism type to two possibilities.  Furthermore, an error is present in that calculation, which we correct in \lref{betaval}.

To prove \tref{topthm}, we use the data of a principal \(S^1\) bundle, namely the base and the first Chern class, to compute the diffeomorphism invariants used by Hambleton and Su for the total space. One, the second Betti number, is calculated easily. When the total space is non-spin but has a spin universal cover, we show how the other invariant  can be computed by applying a map from  \(\Omega_4^{\text{Spin}^c}\to\Omega_4^{\text{Pin}^+}\) to the base.  While a two-fold ambiguity remains in determining which diffeomorphism type corresponds to a specific first Chern class, we are nonetheless able to determine which pairs of invariants are achieved, and achieved infinitely many times, by bundles over a given 4-manifold.  

The paper is organized as follows.  In \sref{diffeo} we examine \(S^1\) actions on 5-manifolds with \(\pi_1=\f{Z}_2\) and prove \tref{topthm}.  In \sref{ettta} we discuss the \(\eta\) invariant of a spin\(^c\) Dirac operator and show that it can be used to detect connected components of the moduli space in the context of \tref{thm}.  In \sref{proof} we compute \(\eta\)  in the case of certain \(4n+1-\)manifolds admitting free \(S^1\) actions and prove \tref{thm}.  In \sref{mandc} we construct the metrics and connections used in the computations of \sref{proof}.    

 I would like to acknowledge my PhD advisor Wolfgang Ziller for all his help, as well as Anand Dessai, David Gonz\'{a}lez-\'{A}lvaro, Fernando Galaz-Garc\'{i}a, and Diego Corro for helpful discussions.  I am further grateful to Yang Su for pointing out how to work around an error in  \cite{HS} and for other useful insights.  

\section{\(S^1\) actions on 5 manifolds with \(\pi_1=\f{Z}_2\)}\label{diffeo}
Our methods for constructing metrics with positive Ricci curvature and for calculating \(\eta\)  use the structure of a principal \(S^1\) bundle.  In this section we prove \tref{bases}, which classifies 5-manifolds with \(\pi_1=\f{Z}_2\) admitting one, or infinitely many, free \(S^1\) actions with simply connected quotients.  \tref{bases} also identifies those quotients.  In particular, we prove \tref{topthm} and show that a manifold \(M^5\) satisfying the hypotheses of \tref{thm} admits infinitely many inequivalent \(S^1\) actions with the same quotient. Our proof relies on a diffeomorphism classification of 5-manifolds with fundamental group \(\f{Z}_2\) carried out by Hambleton and Su \cite{HS}. 

Given a manifold \(M\) with \(\pi_1(M)=\f{Z}_2\), a characteristic submanifold  \(P\subset M\) is defined as follows.  For \(N\) sufficiently large let \(f:M\to\ar{}P^N\) be a classifying map of the universal covering \(\tilde{M}\to M\).  We can choose \(f\) to be transverse to \(\ar{}P^{N-1}\), and hence \(P=f^{-1}(\ar{}P^{N-1})\) is a smooth manifold.  One checks that any two manifolds defined in this way are cobordant.

Alternatively, assume that \(P\subset M\) is a submanifold such that the inverse image \(\tilde P\subset \tilde M\) under the universal covering splits \(\tilde M\) into two components \(\tilde M_1\) and \(\tilde M_2\).   Furthermore \(\partial \tilde M_1=\partial \tilde M_2=\tilde P\) and the covering transformation acting on \(\tilde M\) switches \(\tilde M_1\) and \(\tilde M_2.\)  One can then construct a map \(f:M\to\ar{}P^N\) such that \(P=f^{-1}(\ar{}P^{N-1}).\)   For details see \cite{GT} and \cite{L}.

The key invariant of the classification in \cite{HS} is the class of \(P\) in an appropriate cobordism group.  The appropriate structure on \(P\) depends on the second Stiefel-Whitney classes \(w_2\) of \(M\) and \(\tilde M\).  Hambleton and Su use the following labels for a manifold \(M\) with \(\pi_1(M)=\f{Z}_2\) and universal cover \(\tilde M\):

\[\begin{array}{cc}
\text{Type I} & w_2(T\tilde M)\neq0 \\\text{Type II}&  w_2(TM)=0\\\text{Type III}&  w_2(TM)\neq 0\text{ and }w_2(T\tilde M)=0.\end{array}\]  

\noindent A characteristic submanifold \(P\) of a Type III manifold admits a pin\(^+\) structure, and all such \(P\) are pin\(^+\) cobordant.  Here Pin\(^\pm(n)\) is the extension of \(O(n)\) by \(\f{Z}_2\) such that a preimage of a reflection squares to \(\pm1\) and \(\Omega^{\text{Pin}^\pm}_n\) is the cobordism group of \(n-\)manifolds with pin\(^\pm\) structures.  For details, see  \cite{HS}, \cite{GT}. 

We review the construction of a pin\(^+\) structure on \(P\) as we will use it later.  Let \(\mu=\tilde M\times_{\f{Z}_2}\ar{}\) be the unique nontrivial real line bundle over \(M.\)  Recall that \(\tilde {M}=\tilde M_1\cup_{\tilde{P}}\tilde M_2\) and the covering transformation exchanges the components. Thus the normal bundle \(N\tilde P\) of \(\tilde{P}\) is trivial and the covering transformation reverses the orientation of the fibers.   The normal bundle \(NP\) of \(P\) satisfies
\[NP=N\tilde P/\f{Z}_2\cong\tilde P\times _{\f{Z}_2}\ar{}=\mu|_P.\]
Since \(M\) is orientable, 
\[w_1(NP)=w_1(TP)=w_1(\text{det}(TP))\]
so \(NP\cong \text{det}(TP).\)
Thus
\begin{equation}\label{pbuniso}(TM\oplus2\mu)|_P=TP\oplus 3NP=TP\oplus 3\text{det}(TP).\end{equation}
Using \cite{GT} Lemma 9 and \cite{HS} Lemma 2.3, one checks that \(w_2(TM\oplus 2\mu)=0\). We can apply \cite{KT} Lemma 1.7 to see that a spin structure on \(TP\oplus 3\text{det}(TP)\) induces a pin\(^+\) structure on \(TP.\)    A similar argument on a cobordism shows that any two characteristic submanifolds are pin\(^+\) cobordant.  


Let \(b_2(M)\) denote the second Betti number of a manifold \(M.\) The main theorem for Type III manifolds in \cite{HS} is Theorem 3.1:

\begin{thm}\label{hs1}\cite{HS}
	Let \(M_1,M_2\) be Type III 5-manifolds such that \(\pi_1(M_i)\cong\f{Z}_2\) acts trivially on \(\pi_2(M_i)\) and \(H_2(M_i,\f{Z})\) is torsion free for \(i=1,2\).  Then \(M_1\) is diffeomorphic to \(M_2\) if and only if 
	
	\[b_2(M_1)=b_2(M_2)\text{\ \  and \ \  }
	[P_1]=\pm[P_2]\in\Omega^{\text{\normalfont Pin}^+}_{4}\]
	where \(P_i\) is a characteristic submanifold of \(M_i.\)
\end{thm}

We will take the data of a principal \(S^1\) bundle, namely  the base and the first Chern class, and identify the diffeomorphism type of the total space.  In particular, we will identify when the total space satisfies the hypotheses of \tref{hs1}, and then compute \(b_2\) and \([P]\).  That computation combined with the classification of Type I and II total spaces in \cite{HS} Theorems 6.5 and 6.8 finishes the proof of \tref{bases}, which in turn implies \tref{topthm}.

A straightforward computation using the long exact homotopy and Gysin sequences proves the following; see for instance \cite{HS} Proposition 6.1.

\begin{lem}\label{buntop}
	Let \(B^n\) be a simply connected manifold and let \(M^{n+1}\to B^n\) be a non-trivial principal \(S^1\) bundle with first Chern class \(kd\), where \(d\) is a primitive element of \(H^2(B,\f{Z})\) and \(k\neq0\) is an integer.   Then \(M\) is orientable, \(H_2(M,\f{Z}) \) is torsion free and \(b_2(M)=b_2(B)-1\). \(\pi_1(M)\cong \f{Z}_k\) is generated by any \(S^1\) fiber and acts trivially on \(\pi_2(M)\).  The universal cover of \(M\) is the total space of an \(S^1\) bundle over \(B\) with first Chern class \(d\).  If \(k=2\),  \(M\) is type III if and only if  and \(w_2(TB)=d\) {\normalfont mod 2}.
\end{lem}   

The condition \(w_2(TB)=d\) mod 2 implies the existence of a spin\(^c\) structure on \(B\).  We call \(d\) the canonical class of that spin\(^c\) structure.   On a simply connected manifold a spin\(^c\) structure is uniquely determined by its canonical class.   Thus in the Type III case, given a simply connected spin\(^c\) 4-manifold \(B^4\) with primitive canonical class \(d\), we want to know the diffeomorphism type of the total space \(M^5\) of the \(S^1\) bundle over \(B^4\) with first Chern class \(2d.\)  Since \(b_2(M)\) is  determined by \lref{buntop}, it remains to find the pin\(^+\) cobordism class of a characteristic submanifold \(P^4\subset M^5\).  In fact, the spin\(^c\) structure on \(B^4\) will naturally induce a pin\(^+\) structure on \(P^4\).  

To see this let \(\rho:M\to B\) be the bundle map and let \(\lambda\to B\) be a complex line bundle with first Chern class \(d.\)   \(\rho^*d\) is the unique nontrivial torsion element of \(H^2(M,\f{Z})\).
Let \(\mu\to M\) be the unique nontrivial real line bundle over \(M.\)  As in the proof that a characteristic submanifold of \(M\) will admit a pin\(^+\) structure, see  \cite{GT} Lemma 9 and \cite{HS} Lemma 2.3, \(w_2(\mu\oplus\mu)=w_1(\mu)^2\neq0\). So \(\mu\oplus\mu\) with its natural orientation is a nontrivial complex line bundle.  Since \(\mu\otimes\mu\) is trivial, \(c_1(\mu\oplus\mu)\) is torsion, and we conclude that \(\rho^*\lambda\cong\mu\oplus\mu.\)  

The \(S^1\) action on \(M\) splits \(TM\) into a horizontal bundle isomorphic to \(\rho^*TB\) and a vertical bundle, trivialized by an action field, which we call \(TS^1.\)  The spin\(^c\) structure on \(B\) is equivalent to a spin structure on \(TB\oplus\lambda.\)  That spin structure induces a spin structure on 
\begin{equation}\label{buniso}\rho^*(TB\oplus\lambda)\oplus TS^1\cong TM\oplus \mu\oplus\mu\end{equation}
 and in turn a pin\(^+\) structure on  \(P\subset M\) using \eqref{pbuniso}.  Denote by \(\beta(B,d)\in\Omega_4^{\text{Pin}^+}\) the cobordism class of \(P\) with this pin\(^+\) structure.  We synthesize the construction with the results of \lref{buntop} as follows:
 
 \begin{lem}\label{synth}
 	Let \(B^4\) be a simply connected 4-manifold and let \(M^5\) be the total space of a principal \(S^1\) bundle over \(B\) with first Chern class \(2d\in\zco{B}{2}\) where \(d\) is a primitive element such that \(w_2(TB)=d\) {\normalfont mod 2}.  Then \(M\) satisfies the condition of \tref{hs1} with \(b_2(M)=b_2(B)-1\) and \([P]=\beta(B,d).\)
 	\end{lem}

  In the next lemma, we will see that \(\beta\) is a spin\(^c\) cobordism invariant whenever it is defined.  
 \begin{lem}\label{betaprop} 	
 	Let \(B_1,B_2\) be spin\(^c\) manifolds with primitive canonical classes \(d_1,d_2\) respectively.
 	\begin{itemize}
 		\item[a)] \(\beta(B_1\amalg B_2,d_1+d_2)=\beta(B_1,d_1)+\beta(B_2,d_2)\)
 		
 		\item[b)]If \(B_1\) is spin\(^c\) cobordant to \(B_2\) then \(\beta(B_1,d_1)=\beta(B_2,d_2).\)
 		
 	\end{itemize}
 \end{lem}
  
  \begin{proof}
  	Part \(a\) follows immediately since the total space of the relevant bundle and the characteristic submanifold of that total space will be disjoint unions.  
  	
  	To prove part \(b,\) let \(W\) be a simply connected spin\(^c\) cobordism between \(B_1\) and \(B_2\) with canonical class \(d.\)  The \(d|_{B_i}=d_i\) for each \(i=1,2,\) and \(d\) must be a primitive class.  Let \(\pi:N\to B\) be the principal \(S^1\) bundle over \(W\) with first Chern class \(2d.\)  By \lref{buntop} \(\pi_1(N)=\f{Z}_2.\) \(\partial N=\pi^{-1}(B_1)\amalg \pi^{-1}(B_2)\) and \(M_i=\pi^{-1}(B_i)\to B_i\) is the principal \(S^1\) bundle with first Chern class \(2d_i.\)  
  	
  	Let \(f:N\to \ar{}P^N\) be a classifying  map for the universal cover of \(N\) which is transverse to \(\ar{}P^{N-1}.\)  By \lref{buntop}, \(\pi_1(N)\) is generated by any \(S^1\) orbit, so  \(\pi_1(M_i)\to\pi_1(N)\) is an isomorphism, and \(f|_{M_i}\)  is a classifying map for the universal cover of \(M_i.\)  Thus \(P_i=f^{-1}(\ar{}P^{N-1})\cap M_i\) is a characteristic submanifold of \(M_i\) and \(f^{-1}(\ar{}P^{N-1})\) is a cobordism between \(P_1\) and \(P_2\).  The argument before \lref{synth} proves that the spin\(^c\) structure on \(W\) induces a pin\(^+\)  structure on \(f^{-1}(\ar{}P^{N-1})\).  That pin\(^+\) structure restricts to the pin\(^+\) structures induced on \(P_i\) by the spin\(^c\) structures on \(B_i\).  To see this one must simply note that the nontrivial real line bundle over \(N\) restricts to the nontrivial real line bundle over \(M_i.\)  We conclude that 
  	\[\beta(B_1,d_1)=[P_1]=[P_2]=\beta(B_2,d_2).\]
  \end{proof}
  
  We now see that \(\beta\) defines a map between the spin\(^c\) and pin\(^+\) cobordism groups.  The 4 dimensional spin\(^c\) cobordism group  \(\Omega_4^{\text{Spin}^c}\) is isomorphic to \(\f{Z}^2.\) The isomorphism takes a spin\(^c\) manifold \(B\) with canonical class \(d\) to the characteristic numbers  
  \[\left<d^2,[B]\right>\text{ and    }\ov{8}\left(\left<d^2,[B]\right>-\text{sign}(B)\right).\]
  Here sign\((B)\) is the signature, and the second integer is the index of the spin\(^c\) Dirac operator, which we denote by ind\((B,d).\)  See \cite{BaG},  \cite{S} for details.  To construct generators of \(\Omega_4^{\text{Spin}^c}\) let \(x\in H^*(\f{C}P^2,\f{Z})\) be the generator which is the first Chern class of the Hopf bundle.  Give \(X=\f{C}P^2\) the  spin\(^c\) structure with canonical class \(x\) and  \(Y=\f{C}P^2\#\f{C}P^2\#\overline{\f{C}P^2}\)  the spin\(^c\) structure with canonical class \(d_Y=(3x,x,x)\in H^2(Y,\f{Z})\cong \oplus^3H^2(\f{C}P^2,\f{Z}).\)  Then \([X],[Y]\in \Omega^{\text{Spin}^c}_4\) represent \((1,0)\) and \((9,1)\) under the isomorphism with \(\f{Z}^2\) and form a minimal generating set of \(\Omega^{\text{Spin}^c}_4.\)  Since \(X\) and \(Y\) have primitive canonical classes, and their inverses in the cobordism group are given by reversing orientation, we conclude that every class in \(\Omega^{\text{Spin}^c}_4\) can be represented by a simply connected manifold \(B\) with primitive canonical class \(d\).    \lref{betaprop} implies that by mapping the cobordism class of such a pair to  \(\beta(B,d)\) we can define a homomorphism \(\beta:\Omega^{\text{Spin}^c}_4\to \Omega_4^{\text{Pin}^+}.\)

  Using the isomorphism \(\Omega_4^{\text{Pin}^+}\cong\f{Z}_{16}\) generated by a  pin\(^+\) structure on \(\ar{}P^4\) we prove the following :
  
  \begin{lem}\label{betaval}
  \[\beta(B,d)=\left<d^2,[B]\right>+4\ep\ \text{\normalfont ind}(B,d)\text{\normalfont \ mod\  } 16\]
for an unkown sign \(\ep=\pm1.\) 
  \end{lem}
  
  This lemma corrects a mistake in the statement of Theorem 6.7 in \cite{HS}.  Our argument uses ideas from the proof in \cite{HS} as well as corrections suggested to the author by Yang Su.  
   
  \begin{proof}
  	We will see that \(\beta(X,x)=1\) and \(\beta(Y,d_Y)=5\) or 13.  The lemma then follows since \(\beta\) is a homomorphism and  \(\Omega_{4}^{\text{Spin}^c}\cong\f{Z}^2\).   
  
The principal \(S^1\) bundle \(\ar{}P^5\to \f{C}P^2\) which is a \(\f{Z}_2\) quotient of the Hopf bundle has first Chern class \(2x.\)  Since \(\ar{}P^4\) is a characteristic submanifold of \(\ar{}P^5,\) it follows that 
\[\beta(X,x)=[\ar{}P^4]=1\in\Omega_4^{\text{Pin}^+}.\]  

The second calculation is  more involved.  We use the notation \([z_0,z_1,z_2]\in\f{C}P^2\) and \([z_0,z_1,z_2]_\pm\in\ar{}P^5\) for the respective images of the point \((z_0,z_1,z_2)\in S^5\subset \f{C}^3\).    Let \(\rho:M\to Y\) be the principal \(S^1\) bundle with first Chern class \(2d_Y\in\zco{Y}{2}\) as defined above.  By \lref{buntop}, the double cover \(\tilde M\)  of \(M\) is the total space of a principal \(S^1\) bundle \(\tilde \rho:\tilde M\to Y\) with first Chern class \(d_Y.\)  Let \(g:Y\to \f{C}P^2\) be a classifying map for \(\tilde \rho\) which is transverse to \(\f{C}P^1\subset \f{C}P^2\) and has a regular value \([1,0,0]\in\f{C}P^1.\)  Then \(g^*x=d_Y\) and the pullback of \(\pi:\ar{}P^5\to \f{C}P^2\) by \(f\) has first Chern class \(2d_Y.\)  There is a map of principal \(S^1\) bundles  \(f:M\to \ar{}P^5\) 	covering \(g,\) that is, an \(S^1\) equivariant map making the following diagram commute: 
\[\begin{tikzcd}
M\arrow{r}{f}\arrow{d}{\rho}& \f{R}P^5\arrow{d}{\pi}\\
Y\arrow{r}{g}&\f{C}P^2
\end{tikzcd}\] 

Since the fundamental groups of \(M\) and \(\ar{}P^5\) are generated by \(S^1\) orbits (see \lref{buntop}), \(f_*:\pi_1(M)\to\pi_1(\ar{}P^5)\) is an isomorphism and \(f\) is a classifying map for the double cover \(\tilde M\to M.\)  Thus if we show that \(f\) is transverse to \(\ar{}P^4\subset \ar{}P^5,\) we we can conclude that \(P=f^{-1}(\ar{}P^4)\) is a characteristic submanifold of \( M.\)  Then given the correct pin\(^+\) structure on \(P,\) \(\beta(Y,d_Y)=[P]\in\Omega_4^{\text{Pin}^+}.\)

To see that \(f\) is transverse to \(\ar{}P^4=\{[z_0,z_1,r]_\pm\in\ar{}P^5|r\in\ar{}\}\) note that at points in \(\pi^{-1}(\f{C}P^2\backslash \f{C}P^1),\) \(\ar{}P^4\) is transverse to the \(S^1\) orbits, which are contained in the image of the equivariant map \(f.\)  At points in \(\pi^{-1}(\f{C}P^1)\), we  associate the horizontal space of the \(S^1\) action with \(T\f{C}P^2\).  By assumption on \(g\), \(f\) is transverse to \(T\f{C}P^1,\) and \(T\f{C}P^1\subset T\f{R}P^4.\)  

For later, we also note that \(f\) is transverse to \(\f{R}P^2=\{[z_0,r,0]\in\f{R}P^5|r\in\ar{}\}\) since \(T\f{C}P^1\subset T\ar{}P^2\) except at \([1,0,0]\) which is a regular value of \(f\) by assumption on \(g.\)

  There is a short exact sequence 
\begin{equation}\label{phi}0\to\f{Z}_2\to\Omega^{\text{Pin}^+}_4\xrightarrow{\phi}\Omega^{\text{Pin}^-}_2\to0\end{equation}
where \(\phi\) is given by taking the cobordism class of a submanifold dual to \(w_1^2\); see \cite{HS} p.172 and \cite{KT} p.217 for details.  Thus \(\Omega_2^{\text{Pin}_-}\) is isomorphic to \(\f{Z}_8\) with generator \([\ar{}P^2].\) We  now compute \(\phi([P])=5\), which restricts the possible values of \(\beta(Y,d_Y)=5\) or \(13\) as desired.

We need to find a submanifold of \(P\) dual to \(w_1^2(TP).\)  	Denote by \(N\ar{}P^4\) the normal bundle of \(\ar{}P^4\) in \(\rp{5}\)  and by \(NP\) the normal bundle of \(P\) in \(M\).  Then \(f^*N\rp{4}=NP.\)  Since \(\rp{5}\) and \(M\) are orientable,
\[w_1(TP)=w_1(NP)=f^*w_1(N\rp{4})=f^*w_1(T\rp{4}).\]

Since \(w_1(T\rp{4})^2\) is dual to \(\ar{}P^2\subset\ar{}P^4,\) as long as the mod 2 degree of \(f:f^{-1}(\ar{}P^2)\to\ar{}P^2\) is 1, it follows that \(f^{-1}(\ar{}P^2)\) is dual to \(w_1(TP)^2.\)  For convenience let \(\Sigma=f^{-1}(\ar{}P^2)\).  Since \([1,0,0]\) is a regular point of \(g,\) \([1,0,0]_\pm\) is a regular point of \(f,\) and the degree of \(f\) is the same as the degree of \(f|_{\Sigma}.\)  The degree of \(f\) is the same as the degree of \(g\). The degree of \(g\) is given by  
\[\left<g^*x^2,Y\right>=\left<d_Y^2,[Y]\right>=9.\]  Thus the mod 2 degree of \(f|_{\Sigma}\) is 1 and 
\(\phi([P])=[\Sigma]\in\Omega_2^{\text{Pin}^-}.\) 

Let \(U\) be a tubular neighborhood of the \(S^1\) orbit  of \([1,0,0]_\pm\) and \(V=\ar{}P^2\backslash U.\)    Since \([1,0,0]\) is a regular value of \(g\) we can choose \(U\) to be made up of regular values of \(f.\)  Then \(f|_{f^{-1}(U)}\) is a covering map.  Since \(f\) maps \(S^1\) fibers to \(S^1\) fibers, \(f_*:\pi_1(f^{-1}(U))\to\pi_1(U)\) is surjective and the covering is trivial.  Thus \(f^{-1}(U)\) is the disjoint union of deg\((f)=9\) copies of \(U\) and \(f^{-1}(U\cap \ar{}P^2)\) is 9 copies of \(U\cap\ar{}P^2.\)  The \(S^1\) orbit of \([1,0,0]_\pm\) is a nontrivial loop in \(\ar{}P^2,\) and \(U\cap\ar{}P^2\) is a tubular neighborhood of that loop, diffeomorphic to \(\ar{}P^2\backslash D^2\) (the Mobius band).  The local inverses to \(f|_{f^{-1}(U)}\) are equivariant embeddings  of the oriented tubular neighborhood \(U\) and are all isotopic.  It follows that the 9 embedding of \(\ar{}P^2\backslash D^2\) making up \(f^{-1}(U\cap\ar{}P^2)\) are all isotopic.  Thus the process by which \(TM\) induces a pin\(^+\) structure on \(P,\) which in turn induces a pin\(^-\) structure on \(\Sigma,\) will induce the same pin\(^-\) structure on each of the 9 copies of \(\ar{}P^2\backslash D^2.\)  

 Since \(\pi(\ar{}P^2)=\f{C}P^1\) and \(\pi(U)\cap\f{C}P^1\) is diffeomorphic to a disc \(D^2\) around \([1,0,0]\) made up of regular values of \(g\),  \(g^{-1}(\pi(U)\cap\f{C}P^1)\) is 9 copies of \(D^2\)  and \(\pi(V)=\f{C}P^1\backslash D^2.\) \(\pi|_{\f{R}P^2}\) is injective away from the orbit of \([1,0,0]_\pm,\) and thus is injective on \(V.\)    It follows that \(\rho\) maps \(f^{-1}(V)\) injectively onto \(g^{-1}(\pi(V)).\)  Thus \(f^{-1}(V)\) is diffeomorphic to \(g^{-1}(\f{C}P^2)\) with 9 discs removed while \(f^{-1}(U\cap\ar{}P^2)\) is 9 copies of \(\ar{}P^2\backslash D^2.\)  In other words,  
\begin{equation}\label{consum}
\Sigma\cong g^{-1}(\f{C}P^1)\#\f{R}P^2\#...\#\ar{}P^2\end{equation}
and the nine summands of \(\ar{}P^2\) all have the same pin\(^-\) structure.  \(\Omega_2^{\text{Pin}^-}\) is generated by \([\f{R}P^2],\) and so it remains to compute the value of \([g^{-1}(\f{C}P^1)]\).  

Let \(\chi=g^{-1}(\f{C}P^2)\).  We will use a general method to define a pin\(^-\) structure called \(r_\chi\) on \(\chi\) and compute \([\chi]\in\Omega_2^{\text{Pin}^-}\) with this structure.  We will then show that \(r_\chi\) is the correct pin\(^-\) structure to use, that is, \(r_\chi\) is compatible under \eqref{consum} with the pin\(^-\) structure used to identify \([\Sigma]\)  with \(\phi([P]),\) which we will call \(r.\)

Consider a simply connected spin\(^c\) 4-manifold \(B\) with canonical class \(d\) and \(\nu\) the complex line bundle with \(c_1(\nu)=d\).  Let \(N\subset B\)  be a smooth submanifold dual to \(d.\)  Then \(\nu|_N\) is isomorphic to the normal bundle of \(N\).  The spin\(^c\) structure on \(B\) is equivalent to a spin structure, called \(s\), on \(TB\oplus \nu.\)  Restricted to \(N,\)  this is a spin structure on 
\(TN\oplus2\nu.\)  The transition functions for \(2\nu\) admit a canonical lift from \(\SO(4)\) to \Spin(4); simply multiply two copies of any lift for the transition functions of \(\nu,\) and the sign ambiguities cancel.  Note that the identity lifts to the identity in this way.  Using this lift, \(s\) induces a spin structure \(s_N\) on \(N.\)         

The spin cobordism class of \(N\) depends only on the spin\(^c\) cobordism class of \(B.\)  To see this, note that the dual to the canonical class of a spin\(^c\) cobordism will be a spin cobordism between the two relevant  submanifolds.  Thus we have a homomorphism 
\[\psi:\Omega_4^{\text{Spin}^c}\to\Omega_2^{\Spin}\cong\f{Z}_2.\]
defined by \(\psi([B])=[N].\)  Indeed, there is a long exact sequence 
\[\to\Omega_4^{\text{Spin}}\to\Omega_4^{\text{Spin}^c}\to\Omega_2^{\text{Spin}}(BU(1))\to\Omega_3^{\Spin}=0\] 
 as in \cite{HS} p.154 and \cite{HKT} p.654. We see that \(\psi\) is surjective by noting that \(\psi\) is the composition of \(\Omega_4^{\text{Spin}^c}\to\Omega_2^{\text{Spin}}(BU(1))\) with the surjective map \(\Omega_2^{\text{Spin}}(BU(1))\to \Omega_2^{\Spin}\) which ignores the map to \(BU(1)\) .  
\

Recall that \(X,Y\) generate \(\Omega^{\text{Spin}^c}_4.\)  The canonical class of \(X\) is dual to \(\f{C}P^1\subset\f{C}P^2\), which is nullcobordant, so \(\psi([X])=0.\)  Since \(\psi\) is surjective, \(\psi([Y])\) generates \(\Omega_2^{\text{Spin}}\).  Since \(\f{C}P^1\) contains a regular value of \(g\), the degree of \(g|_{\chi}\) equals the degree of \(g\) and \(\chi\) is dual to \(g^*x=d_Y\).  Giving \(\chi\) the spin structure \(s_\chi\) used to define \(\psi\),  \(\psi([Y])=[\chi]\neq0.\)
	
	\Spin(n) embeds naturally into both \(\Pin^\pm(n)\), so a spin\  structure induces a natural pin\(^-\) structure.  Kirby and Taylor show that in dimension 2, the corresponding map \[\Omega_2^{\Spin}\cong\f{Z}_2\to\Omega_2^{\Pin^-}\cong\f{Z}_8\]   
is injective, see proposition 3.8 in \cite{KT}.  Let \(r_\chi\) be the Pin\(^-\) structure on \(\chi\) induced by \(s_\chi.\)  Using that structure \([\chi]=4\in\Omega_2^{\text{Pin}^-}.\)  Once we confirm that \(r_\chi\) is the correct structure, we conclude with \eqref{consum} that  \(\phi([P])=5,\) completing the proof of \lref{betaval}.   


Let \(r\) be the pin\(^-\) structure on \(\Sigma\) used to define \(\phi([P]).\)  Recall that \(\rho\) is a diffeomorphism between the open set \(O=f^{-1}(V)\subset \Sigma\) and \(\rho(O),\) which is \(\chi\) with 9 discs removed.  It remains only to check that \(r=\rho^*r_\chi\) on \(O.\)

We first recall the definition   of \(r.\)  Let \(\mu\) be the nontrivial real line bundle over \(M\) and let \(E=TM\oplus 2\mu.\)  Let \(\lambda\) be the complex line bundle over \(Y\) with \(c_1(\lambda)=d_Y\) and let \(s\) be spin structure on \(TY\oplus \lambda\) used in the definition of \(\psi\). With the isomorphism \eqref{buniso},  \(s\) induces a spin structure on \(E\) called \(s_E.\)  Then \eqref{pbuniso} shows
\[E|_P=TP\oplus3\text{det}(TP)\] and  we induce a pin\(^+\) structure on \(TP\) using a canonical lift of the transition functions of \(3\text{det}(TP)\) from \(\O(3)\) to Pin\(^-(3)\).  In turn,
\[TP|_\Sigma=T\Sigma\oplus2\text{det}(T\Sigma)\]
and using a canonical lift of the transition functions of \(2\text{det}(T\Sigma)\) from \(\O(2)\) to \(\Pin^+(2)\) we induce the pin\(^-\) structure \(r\) on \(\Sigma.\)  Note that normal bundle of \(\Sigma\) in \(P\) is orientable and thus 
\[w_1(\text{det}(T\Sigma))=w_1(\text{det}(TP)|_\Sigma)\]
In this way we can combine the two steps and see that \(s_E\) induces \(r\) on \(T\Sigma\) using the isomorphism 
\begin{equation}\label{detiso}E|_\Sigma=T\Sigma\oplus5\text{det}(T\Sigma)\end{equation}
and a canonical lift of the transition functions of 5det\((T\Sigma)\) from \(\O(5)\) to Pin\(^+(5).\)  The details of the canonical lifts involved can be found in \cite{KT} Lemma 1.7; the salient fact is that each lifts the identity to the identity.  

Next, we note that det\((T\Sigma)\) and \(\rho^*\lambda\) are trivial over \(O.\)  The former follows because because \(O\) is an open set in \(\Sigma\), but is orientable since it is diffeomorphic to an open set in \(\chi\).  As for the latter, we have seen that \(\rho^*\lambda\cong 2\mu\), \(\mu|_P=\text{det}(TP),\) and \(\text{det}(TP)|_\Sigma=\text{det}(T\Sigma).\) Since  \(\rho\) is a diffeomorphism on \(O\) and \(\rho^*\lambda\) is trivial, \(\lambda\) is trivial on \(\rho(O).\)

Let \(t_{ij}\) be transition functions with values in \SO(2) for \(T\chi.\)  As we saw in the definition of \(\psi\), for points in \(\chi,\)
\[TY\oplus\lambda\cong T\chi\oplus2\lambda.\]    Thus on \(\rho(O)\) the transition functions for \(\lambda\) can be chosen to be the identity and the transition functions for \((TY\oplus\lambda)|_{\chi}\) can be chosen to be \(t_{ij}.\)  The spin structure \(s\) gives a lift of \(t_{ij}\) to \(\tilde t_{ij}\) in \(\Spin(2).\) Since the canonical lift of the transition functions for \(2\lambda\) will also be the identity, \(\tilde t_{ij}\) is also the lift given by \(s_\chi\) and \(r_\chi.\)

Furthermore, using \eqref{buniso},  \(t_{ij}\circ \rho\) are transition functions for \(E\) on \(O\).  By definition, \(s_E\) gives the lift \(\tilde t_{ij}\circ\rho.\)  Using \eqref{detiso}, \(t_{ij}\circ\rho\) are transition functions for both \(E|_O\) and \(T\Sigma,\) compatible by picking trivial transition functions for \(5\text{det}(T\Sigma)\).  The canonical lift of the transition functions for \(5\text{det}(T\Sigma)\) will also be trivial, and the lift given by \(r\) will simply be the inclusion of \(\tilde t_{ij}\circ \rho\) into Pin\(^-(2).\) Thus \(r=\rho^*r_\chi\) on \(O.\)  This completes the proof of \lref{betaval}.

   \end{proof}
  
  We can now prove \tref{topthm}.  In fact, we prove the following more detailed theorem which includes the statement of \tref{topthm}.  Here we use the notation of Hambleton and Su, where \(\#_{S^1}\) is gluing along the boundary of a tubular neighborhood of a generator of \(\pi_1.\) \(X(q), q=1,3,5,7\) are the 4 closed manifolds homotopy equivalent to \(\ar{}P^5,\) with \(X(1)=\ar{}P^5\), and \(X(q), q=0,2,4,6,8\) are constructed from pairs of homotopy \(\ar{}P^5\)'s using the operation \(\#_{S^1}.\)  The labeling is such that a characteristic submanifold \(P\subset X(q)\) has class \(q\in\Omega_4^{\text{Pin}^+}/\pm=\{0,...,8\}\). See the discussion before Theorem 3.7 in \cite{HS} for details.  
  
  \begin{thm}\label{bases}
  
  			Let \(M\) be a 5-manifold with \(\pi_1=\f{Z}_2.\) Let \(P\subset M\) be a  characteristic submanifold.
  			\begin{itemize}
  			\item[1)]  \(M\) admits a free \(S^1\) action with a simply connected quotient if and only if \(M\) is orientable, \(H_2(M,\f{Z})\) is torsion free, and \(\pi_1(M)\) acts trivially on \(\pi_2(M).\) Furthermore if \(b_2(M)=0\) then  \(M\) is diffeomorphic to \(\ar{}P^5\).
  			\item[2)]Suppose \(M^5\) satisfies the conditions in 1).  Let \(Q(M)\) be the set of quotients of \(M\) by free \(S^1\) actions.  The following table gives necessary and sufficient conditions for a 4-manifold to be in \(Q(M)\). \(S\) is a set of four exceptional Type I 5-manifolds described in the final two rows.  If \(b_2(M)>0\) then for each \(B\in Q(M)\), \(M\) admits infinitely many inequivalent \(S^1\) actions with quotients diffeomorphic to \(B\).  
  			
  			\end{itemize}
  
\end{thm}

  		  \begin{tabular}{|l|l|}
  			\hline
  			&\\
  			
  			\(M^5\)&\(Q(M^5)\)= simply connected 4-manifolds \(B^4\) such that\\\hline Type II&\(B\) is spin and \(\ b_2(B)=b_2(M)+1\)\\Type III& \(B\) is non-spin, \(\ b_2(B)=b_2(M)+1 \text{ and sign}(B)=\pm[P]\text{ mod }4\)\\Type I and \(M\notin S\)&\(B\) is non-spin and  \(b_2=b_2(M)+1\)\\\(X(q)\#_{S^1}(\f{C}P^2\times S^1),\ q=0,4\)&\(B\) is non-spin, \(b_2=3\) and \( |\text{sign}(B)|=1\)\\\(X(q)\#_{S^1}(S^2\times\ar{}P^3)\ q=0,4\)&\(B\) is non-spin,  \(b_2=4\) and \( |\text{sign}(B)|<4\)\\
  			\hline
  		\end{tabular}

  
 \

\bigskip

Thus given \(M^5\) satisfying the hypotheses of 1) and matching the description of one of the rows in the left column, a 4-manifold \(B^4\) is diffeomorphic to a quotient of \(M^5\) by a free \(S^1\) action if and only if it satisfies the conditions given in the corresponding row of the right column.

  \begin{proof}
   We prove 2) first.  Let \(M\) be an orientable 5-manifold with \(\pi_1(M)=\f{Z}_2\) acting trivially on \(\pi_2(M)\), \(H_2(M,\f{Z})\) torsion free, and \(b_2(M)>0\) unless \(M\cong\ar{}P^2\).  Let \(P\subset M\) be a characteristic submanifold.  
   
   Note that if \(M\to B\) is principal \(S^1\) bundle, the long exact homotopy sequence implies that \(\pi_1(M)\to\pi_1(B)\) is surjective.  If \(\pi_1(B)=\f{Z}_2,\) then the Gysin sequence implies that \(H^3(B)\to H^3(M)\) is injective.  Since \(M,\) and thus \(B,\) is orientable, \(H^3(B)=\f{Z}_2\) and \(H_2(M)\)
would not be torsion free.  Thus any quotient of \(M\) by a free \(S^1\) action is simply connected.  
   
  \subsection*{M is Type II}
  
  First, suppose \(M\to B\) is a principal \(S^1\) bundle. By \lref{buntop}, \(b_2(B)=b_2(M)+1\) and by \cite{HS} Proposition 6.1 \(B\) is spin.
    
   Conversely, Let \(B\) be a  simply connected spin 4-manifold with \(b_2(B)=b_2(M)+1.\)   Then by \cite{HS} Proposition 6.1, all of the total spaces of principal \(S^1\) bundle over \(B\) with \(\pi_1=\f{Z}_2\) are Type II and have second Betti number \(b_2(B)-1\). By \cite{HS} Theorem 3.1 all such total spaces are diffeomorphic to \(M.\)  If \(b_2(M)\geq1\) there are infinitely many primitive elements of \(H^2(B,\f{Z})=\f{Z}^{b_2(M)+1}\) and thus infinitely many non-isomorphic such bundles.  
     
    \subsection*{M is Type III}
    
     Suppose \(M\to B\) is a principal \(S^1\) bundle. By \lref{buntop}, \(b_1(B)=b_1(M)+1\) and the first Chern class of the bundle is \(2d,\) where \(d\) is a primitive element of \(H^2(B,\f{Z})\) such that \(w_2(TB)=d\) mod 2.  It follows that \(B\) is non-spin, and by \cite{LM} Corollary II.2.12, p. 89 the intersection form of \(B\) is odd.  By the classification of integral forms and Donaldson's Theorem, \cite{DK} p.5 and Theorem 1.3.1, p.25, the intersection form of \(B\) is diagonal, and so  \(H^*(B,\f{Z})=\zco{\#^a\f{C}P^2\#^b\overline{\f{C}P^2}}{*}\)
  for some integers \(a,b.\)  Then using \cite{LM} Corollary II.2.12 again we see that 
  \[w_2(B)=(1,1,...,1)\in\twoco{B}{2}\cong\f{Z}_2^{a+b}.\]
  Thus \(d=(d_1,...,d_{a+b})\in\zco{B}{2}\cong\f{Z}^{a+b}\) where each \(d_i\) is an odd integer.  This completes the proof of one direction of  2) since
  \[[P]=\beta(B,d)=\left<d^2,[B]\right>=\sum_{i=1}^ad_i^2-\sum_{j=a+1}^bd_j^2=\text{sign}(B) \quad\text{ mod } 4.\]
  
Conversely, Let \(B\) be a non-spin simply connected 4-manifold with \(b_2(B)=b_2(M)+1.\)  Assume further that sign\((B)=[P]\in\f{Z}_4/\pm.\) Again, \(H^*(B,\f{Z})=\zco{\#^a\f{C}P^2\#^b\overline{\f{C}P^2}}{*}\) where \(b_2(B)=a+b\) and sign\((B)=a-b.\)  Choose \(c\in\{0,1,2,3\}\) such that \(\pm[P]=a-b+4c\text{ mod 16}\).  If \(b_2(M)>0\), choose \(k\) such that \[(4+2\ep)k(k+1)=4c\text{ mod }16\]
  where \(\ep=\pm1\) is the sign from \lref{betaval}.  If \(b_2(M)=0\) then choose \(k=0\).  Set
  \[d_k=(1+2k,1,...,1)\in\zco{B}{2}\cong\f{Z}^{a+b}.\] Then \(d\) is primitive and as above, we see that \(w_2(TB)=d\) mod 2.   Using \lref{betaval} we have

  \[\beta(B,d_k)=\text{sign}(B)+(4+2\ep)k(k+1)=\pm[P] \text{ mod }16\]
  Hence by \lref{synth} and \tref{hs1}  \(M\) is diffeomorphic to the total space of an \(S^1\) bundle over \(B\) with first Chern class \(2d_k.\)  In the case where \(b_2(M)>1,\) there are infinitely many choices of \(k\) yielding distinct classes \(d_k,\) and \(M\) is  diffeomorphic to infinitely many total spaces of non-isomorphic \(S^1\) bundles over \(B.\)
  
  \subsection*{M is Type I}

  Suppose \(M\to B\) is a principal \(S^1\) bundle.  By \lref{buntop}, \(b_1(B)=b_1(M)+1\) and by \cite{HS} proposition 6.1 \(B\) is non-spin and and the first Chern class of the bundle is \(2d,\) where \(d\) is a primitive element of \(H^2(B,\f{Z})\) such that \(w_2(TB)\neq d\) mod 2.  
  
  If \(M=X(q)\#_{S^1}(\f{C}P^2\times S^1),\ q=0,4\) then \(b_2(B)=3\) and by \cite{HS} Theorem 6.8 \(\left<d^2,[B]\right>=\pm q\)   mod 8.  If sign\((B)=\pm3,\) then up to orientation as above \(\zco{B}{*}=\zco{\#^3\f{C}P^2}{*}\) and  \(w_2(TB)=(1,1,1).\)  Thus
  \[d=(d_1,d_2,d_3)\in\zco{B}{2}\cong\f{Z}^3.\]
  and some \(d_i\) must be even.  Since \(d\) is primitive, some \(d_i\) must be odd.  One easily checks that under these conditions, \(\left<d^2,[B]\right>\neq0,4\) mod 8.  So sign\((B)=\pm1.\)
  
  If \(M=X(q)\#_{S^1}(S^2\times \ar{}P^3),\ q=0,4\) then \(b_2(B)=4\) and  \(\left<d^2,[B]\right>=\pm q\)   mod 8. If sign\((B)=\pm4,\) then up to orientation by the argument in the Type III case, \(\zco{B}{*}=\zco{\#^4\f{C}P^2}{*}\) and  
  \[d=(d_1,d_2,d_3,d_4)\in\zco{B}{2}\cong\f{Z}^3\]
  with at least one \(d_i\) even and at least one \(d_i\) odd.  Again \(\left<d^2,[B]\right>\neq0,4\) mod 8, so \(|\)sign\((B)|<4.\)

  Conversely, Let \(B\) be a  simply connected non-spin 4-manifold satisfying the conditions given by the table for \( Q(M)\).  Then \(H^*(B,\f{Z})=\zco{\#^a\f{C}P^2\#^b\overline{\f{C}P^2}}{*}\) for some integers \(a,b\) such that \(a+b=b_1(M)+1.\)    Let \((q,s) \in\f{Z}_8\oplus\f{Z}_2\) represent the cobordism class of \(P\subset M\) in the pin\(^c\) cobordism group \(\Omega_4^{\text{Pin}^c}\cong\f{Z}_8\oplus\f{Z}_2\); see \cite{HS} p.154.  By \cite{HS} Theorem 3.6 \(q+s=b_2(M)+1\) mod 2.  
  
  If \(q=0,4\) then  \cite{HS} Theorem 3.7  implies that \(a+b\geq3\), so we can assume that up to orientation \(a\geq2\) and using the table either \(a+b\geq5\) or \(|\)sign\((B)|<b_2(B)\),  which implies \(b>0.\)   Define the following elements \(d_k\in\zco{B}{2}\cong\f{Z}^a\oplus\f{Z}^b\) for each \(k\in\f{Z}.\)
  
  \[\begin{array}{ccc}q=0:&d_k=(1+8k,0,...,0,1)\text{ if \(b>0\)}&d_k=(2+8k,1,1,1,1,0,...,0)\text{ if \(b=0\)}\\q=4:&d_k=(2+8k,1,0,...,0,1)\text{ if \(b>0\)}&d_k=(1+8k,1,1,1,0,...,0)\text{if \(b=0\)}\end{array}.\]
  
  \
  
  If \(q=2\), \cite{HS} Theorem 3.7 implies that \(a+b\geq 3\) and we can assume \(a\geq2\) and define 
  \[ \begin{array}{ccc}q=2:&d_k=(1+8k,1,0,...,0)&\end{array}\]
  
  If \(q\) is odd, By \cite{HS} Theorem 3.7 \(a+b\geq2,\)  and we can assume \(a\geq1\). Define
  
  \[ \begin{array}{ccc}q=1:&d_k=(1+8k,4,0,...,0)&\\q=3:&d_k=(1+8k,2,0,...,0).&\end{array}\]
In each case \(d_k\) is primitive, \(w_2(TB)\neq d_k\) mod 2, and \(q=\pm\left<d_k^2,[B]\right>\) mod 8.  By   \cite{HS} Theorem 6.8 the \(S^1\) bundle over \(B\) with first Chern class \(2d_k\) is diffeomorphic to \(M\).  Again  infinitely many \(k\) yield distinct classes \(d_k\) an thus non-isomorphic bundles.  

  \bigskip

  To prove 1), first assume \(M\) is a 5-manifold with \(\pi_1(M)=\f{Z}_2\)  admitting a free \(S^1\) action with simply connected quotient \(B\).  By \lref{buntop}, \(M\) is orientable, \(\pi_1(M)\) acts trivially on \(\pi_2(M)\) and \(H_2(M,\f{Z})\) is torsion free.  If \(b_2(M)=0\), then \(b_2(B)=1\) and up to orientation \(\zco{B}{*}\cong\zco{\f{C}P^2}{*}\) and \(w_2(TB)\) is non-zero.    There are only two primitive classes  \(\pm d\in H^2(B,\f{Z})\cong\f{Z}\), each restricting to \(w_2(B)\) mod 2.  Thus \(B\) is of Type III and \(\beta([B,d])=\pm1.\)  By \tref{hs1} \(M\) is diffeomorphic to \(\ar{}P^5.\)  
  
      To prove the converse, suppose  \(M\) is an orientable 5-manifold with  \(\pi_1(M)=\f{Z}_2\) acting trivially on \(\pi_2(M)\) and \(H_2(M,\f{Z})\) torsion free.  Let \(P\subset M\) be a characteristic submanifold.  Since \(\ar{}P^5\) admits a free \(S^1\) action induced by the Hopf action we assume \(b_2(M)>0\).   We must  show the set \(Q(M)\) described  in the table in 2) is nonempty. 
      
      If \(M\) is Type II, by \cite{HS} Theorem 3.6 \(b_2(M)\) is odd.  Then  \(B=\#^{(b_2(M)+1)/2}S^2\times S^2\in Q(M).\)       If \(M\) is Type I  then \(B=\#^{b_2(M)}\f{C}P^2\#\overline{\f{C}P^2}\in Q(M)\).  If \(M\) is Type III, let \(0\leq c<16\) be an integer such that \([P]=c\ \text{mod }16.\)  By \cite{HS} Theorem 3.6 we see that \(c=b_2(M)+1\text{ mod }2\).  Choose \(l\) such that 
  \[0\leq c-4l<4.\] Then
  \[a=\frac{b_2(M)+1+c-4l}{2}\quad\text{and}\quad b=\frac{b_2(M)+1-c+4l}{2}\]
  are nonnegative integers.  Let \(B=\#^a\f{C}P^2\#^b\overline{\f{C}P^2}.\)  Then \(b_2(B)=b_2(M)+1\) and sign\((B)=[P]\in\f{Z}_4/\pm.\)  So \(B\in Q(M)\).  \end{proof}

  \
  
  \

  We note that the final paragraph of the proof above in fact shows the following, which we will make use of later.  
  
  \begin{cor}\label{goodbases}
  	Let \(M\) be a 5-manifold with \(\pi_1=\f{Z}_2\) admitting a free \(S^1\) action with a simply connected quotient.  Then \(M\) admits a free \(S^1\) action with quotient diffeomorphic to either \(\#^cS^2\times S^2\) or \(\#^a\f{C}P^2\#^b\overline{\f{C}P^2}\) for some \(a,b,c\in\f{Z}.\)   
  \end{cor}
  Combining \tref{bases} with Theorem 3.7 in \cite{HS} we can characterize the manifolds satisfying \tref{thm}.

  \begin{cor}\label{thmmfds}Let \(M^5\) be a 5-manifold.  The following are equivalent:
  	
  	\begin{itemize}
  		\item[1)] \(M^5\) is Type III and admits a free \(S^1\) action with a simply connected quotient. 
  		\item [2)]There exists \(B^4=\cpcs{a}{b}\),  \(a,b\in\f{Z}_{\geq0}\) such that \(M^5\) is the total space of a principal bundle \(S^1\to M^5\to B^4\)  with first Chern class 2d, where \(d\in H^2(B^4,\f{Z})\) is  primitive and \(w_2(TB^4)=d\) mod 2.
  		\item[3)] There exits \(k\in\f{Z}_{\geq0}\) and \(q\in\{0,1,...,8\}\), with \(k>0\) if \(q\) is 3,5, or 7, such that \(M^5\) is diffeomorphic to \[X(q)\#_{S^1}(\#^k(S^2\times S^2)\times S^1).\]
  	\end{itemize}
  	
  	If those conditions are satisfied and then \(\mods{Ric}(M^5)\) has infinitely many path components.  
  \end{cor}
  
  \begin{proof}
  	 1) implies 2) by \lref{buntop} and \cref{goodbases}  .   If we assume 2), \lref{buntop} implies that \(M\) is a Type III manifold with \(\pi_1\) acting trivially on \(\pi_2\) and \(H_2(M,\f{Z})\) torsion free.  Theorem 3.7 in \cite{HS} shows that every such manifold is diffeomorphic to   	\(X(q)\#_{S^1}(\#^k(S^2\times S^2)\times S^1)\) for some \(q\in\{0,...,8\}\) and some \(k\in\f{Z}_{\geq0}\).  If \(k=0\) and \(q\) is odd, \(b_2(X(q))=0\) and using \tref{bases} \(M\) must be diffeomorphic to \(\ar{}P^5=X(1).\)  
  	 
  	 	By Theorem 3.7 in \cite{HS}, \(M=X(q)\#_{S^1}(\#^k(S^2\times S^2)\times S^1)\) is an orientable Type III manifold with \(\pi_1(M)\) acting trivially on \(\pi_2(M)\),   \(H_2(M,\f{Z})\) torsion free, and \(b_2(M)=2k+(1+(-1)^q)/2\).  Thus by \tref{bases}  3) implies 1).  
  	
  	Now assume \(M\) satisfies the conditions.  If \(b_2(M)>0\), then by \lref{buntop} the integers \(a,b\) in 2) must satisfy \(a+b\geq2\).   Then \tref{thm} implies that \(\mods{Ric}(M)\) has infinitely many path components.  By \tref{bases}, if \(b_2(M)=0,\) then 1) implies that \(M\cong \ar{}P^5.\)  \(\mods{Ric}(\ar{}P^5)\) is shown to have infinitely many path components in \cite{DG}.  
  \end{proof}

  \begin{rem}
  	
  	\
  	
  	\begin{itemize}
  		\item By the discussion proceeding Theorem 3.7 in \cite{HS} the manifolds described in \cref{thmmfds}  can also be constructed by applying \(\#_{S^1}\) to the homotopy \(\ar{}P^5\)'s.  For instance
  		\[X(q)\#_{S^1}(\#^k(S^2\times S^2)\times S^1)\cong X(q)\#_{S^1}X(0)\#_{S^1}...\#_{S^1}X(0).\]
  		\item It is  shown in \cite{DG} that \(\mods{Ric}\) also has infinitely many components for the homotopy \(\ar{}P^5\)'s \(X(3),X(5)\) and \(X(7)\) .
  		\item   A characteristic submanifold \(P\subset X(q)\#_{S^1}(\#^k(S^2\times S^2)\times S^1)\) has class \(q\in\Omega_4^{\text{Pin}^+}/\pm=\{0,...,8\}\).  If we fix a non-spin simply connected 4-manifold \(B^4\), then a Type III total space of a principal \(S^1\) bundle over \(B\) will be diffeomorphic to \(X(q)\#_{S^1}(\#^k(S^2\times S^2)\times S^1).\) Using the table in \tref{bases} we see that \(q\) must satisfy \(q=\pm \text{sign}(B)\) mod 4.  It follows that there are 2,3 or 4 choices of \(q\), and the same number of diffeomorphism types of Type III total spaces, if sign\((B)\) is 2, 0, or \(\pm\)1 mod 4 respectively.  The value of \(q\) can be determined, up to two possibilities, using \lref{betaval}.  The set of diffeomorphism types of Type I total spaces is more complicated, but can be computed using \tref{bases} and Theorem 3.7 in \cite{HS}.  If \(B^4\) is a simply connected spin 4-manifold there exists a unique diffeomorphism type of total spaces with \(\pi_1=\f{Z}_2\) , represented by \((S^2\times \ar{}P^3)\#_{S^1}(\#^{(b_2(B)-2)/2}(S^2\times S^2)\times S^1).\) 
  	\end{itemize}
 
  \end{rem}

Using a result of Gilkey, Park and Tuschmann, we can lift metrics from the quotients described by \cref{goodbases} to prove the following:

\begin{cor}\label{onemetric}
	Let \(M\) be a 5 manifold with \(\pi_1(M)=\f{Z}_2\) admitting a free \(S^1\) action with a simply connected quotient.  Then \(M\) admits a metric with positive Ricci curvature.  
\end{cor}

\begin{proof}
 In \cite{SY1}  Sha and Yang put a metric of positive Ricci curvature on \(\#^cS^2\times S^2.\)  A modification of Perelman's construction in \cite{P} puts such a metric on \(\cpcs{a}{b}\), see \lref{per}.  \cref{goodbases} shows that \(M^5\) admits a free \(S^1\) action with quotient \(B^4\) diffeomorphic to one of those manifolds.   Gilkey, Park, and Tuschmann \cite{GPT} showed that if \(B^4\) admits Ric\(>0,\) \(M^5\) is the total space of a principal bundle over \(B^4\) with compact connected structure group \(G,\) and \(\pi_1(M^5)\) is finite, then \(M\) admits a \(G\) invariant metric with Ric\(>0.\) In this case \(G=S^1\), \(\pi_1(M)=\f{Z}_2\) and the corollary follows.  
\end{proof}

The corresponding result in the simply connected case was proved by Corro and Galaz-Garcia in \cite{CG}.  By Lichnerowicz's theorem, many simply connected 4-manifolds, such as a K3 surface, do not admit even positive scalar curvature .  It is interesting to note that \cref{onemetric} and the results of \cite{CG}  imply that total spaces with \(\pi_1=0\) or \(\f{Z}_2\) of principal \(S^1\) bundles over such manifolds nonetheless admit metrics of positive Ricci curvature. 

\section{\(\eta\) Invariant }\label{ettta}
We use the \(\eta\) invariant of the spin\(^c\) Dirac operator, which we define in this section, to distinguish components of geometric moduli spaces.  A manifold \(M\) is spin\(^c\) if there exists a  complex line bundle \(\lambda\) over \(M\) such that the frame bundle of \(TM\oplus\lambda\), a principal \(SO(n)\times U(1)\) bundle, lifts to a principal \(\text{Spin}^c(n)=\text{Spin}(n)\times_{\f{Z}_2}U(1)\) bundle.  A manifold is spin\(^c\) if and only if the second Stiefel-Whitney class \(w_2(TM)\) is the image of an integral class \(c\in H^2(M,\f{Z})\) under the map \(H^2(M,\f{Z})\to H^2(M,\f{Z}_2).\)  In this case \(c\), which we call the canonical class of the spin\(^c\) structure, is the first Chern class of \(\lambda\), which we call the canonical bundle.  

Using complex representations of Spin\(^c(n)\) we form spin\(^c\) spinor bundles and equip them with actions of the complex Clifford algebra bundle \(\f{C}l(TM)\).  When the dimension of \(M\) is even there is a unique irreducible such bundle \(S\) with a natural grading \(S=S^+\oplus S^{-}\).  Given a metric \(g\) on \(M\) and a unitary connection \(\nabla\) on \(\lambda\) we can construct a spinor connection \(\nabla^s\) on \(S,\) compatible with Clifford multiplication, and  a spin\(^c\) Dirac operator \(D^c_{g,\nabla}\) acting on sections of \(S.\)  See \cite{LM} Appendix D for details.  The Bochner-Lichnerowicz identity for this operator is
\begin{equation}\label{boch}
(D^c_{g,\lambda})^2=(\nabla^s)^*\nabla^s+\ov{4}\text{scal}(g)+\frac{i}{2}F^\nabla
\end{equation}

\noindent where the complex two-form \(F^\nabla\) is the curvature of \(\nabla\).  This form acts on the spinor bundle \(S\) by way of the vector bundle isomorphism \(\Lambda T^*M\to\Lambda TM\to\f{C}l(TM)\) given by \(g\).  
The operator \((\nabla^s)^*\nabla^s\) is nonnegative definite with respect to the \(L^2\) inner product on a closed manifold or a compact manifold with boundary on which the Atiyah-Patodi-Singer boundary conditions have been applied.  See \cite{APSII} Theorem 3.9 for details.    The remaining term \(\ov{4}\text{scal}(g)+\frac{i}{2}F^\nabla\) is positive definite if 
\begin{equation}\label{curvs}\text{scal}(g)>2|F^\nabla|_g,\end{equation}
where the norm \(|\cdot|_g\) is the operator norm on \(\f{C}l(TM)\) acting on \(S.\) In particular, ker\((D^c_{g,\nabla})=0\) if \eqref{curvs} is satisfied.   For a later purpose we note that for \(\omega\in\Omega^2(M,\f{C})\) and an orthonormal basis \(\{e_i\}\) of \(TM\) with respect to \(g,\) we have 
\begin{equation}\label{norm}
|\omega|_g\leq\sum_{i<j}|\omega(e_i,e_j)|.
\end{equation}

Suppose  \(W\) is a spin\(^c\) manifold with boundary \(\partial W=M,\) with \(\lambda\) and \(c\) defined on \(W\) as above.  W induces a spin\(^c\) structure on \(M\) with canonical class \(c|_{\partial W}\) and canonical bundle \(\lambda|_{\partial W}.\)  Choose a metric \(h\) on \(W\) and a connection \(\nabla\) on \(\lambda\) which are product-like near \(\partial W\), i.e.
\[h=h|_{\partial W}+dr^2\] and 
\[\nabla=\text{proj}_{M}^*(\nabla|_{\partial W})\]
on a collar neighborhood \(U\cong M\times I\) where \(I\) is an interval with coordinate \(r.\)  Applying the Atiyah-Patodi-Singer boundary conditions, the Atiyah-Patodi-Singer index theorem \cite{APS} states that
\begin{equation}\label{aps}\text{ind}(D^c_{h,\nabla}|_{S^+})=\int_We^{c_1(\nabla)/2}\hat{A}(p(g))-\frac{\text{dim}(\text{ker}(D^c_{h|_{\partial{W}},\nabla|_{\partial{W}}}))+\eta(D^c_{h|_{\partial{W}},\nabla|_{\partial{W}}})}{2}.\end{equation}
Here \(c_1(\nabla)\) and \(p(g)\) are the Chern-Weil Chern and Pontryagin forms constructed from the curvature tensors of the connection and metric respectively.  \(\hat{A}\) is the polynomial in the Pontryagin forms and   \(D^c_{h|_{\partial{W}},\nabla|_{\partial{W}}}\) is the spin\(^c\) Dirac operator on \(M\) constructed using the induced metric and connection. 

\(\eta\) is an analytic invariant of the spectrum of an elliptic operator defined in \cite{APS}.  Given an elliptic differential operator \(D\)  with spectrum \(\{\lambda_i\}\) we define a complex function
\[\eta(D,s)=\sum_{\lambda_i\neq0}\text{sign}(\lambda_i)|\lambda_i|^{-s}.\]

\noindent One shows that the function is analytic when the real part of \(s\) is large and  Atiyah, Patodi and Singer showed that it can be analytically continued to a meromorphic function which is analytic at 0.  Thus we define \(\eta(D)=\eta(D,0).\)  If a diffeomorphism \(\phi\) preserves the spin\(^c\) structure, then \(D^c_{\phi^*g,\phi^*\nabla}\) is conjugate to \(D^c_{g,\nabla}\) and hence they have the same spectrum and the same values of \(\eta.\)  We will use \eqref{aps} to calculate \(\eta\) for an operator \(D_{g,\bar\nabla}\) on a manifold \(M\) by finding a suitable \(W\) with \(\partial W=M\) and extending \(g\), \(\bar\nabla\) to product like \(h\) and \(\nabla\) on W.   

  Kreck and Stolz combined the \(\eta\) invariant with information about the Chern-Weil forms of the metric to get an invariant for metrics on \(4n+3\) dimensional spin manifolds.  We prove that the \(\eta\) invariant alone provides the desired invariant for certain \(4n+1\)  dimensional spin\(^c\) manifolds.

\begin{thm}\label{etta}
Let \(M^{4n+1}\) be a closed spin\(^c\) manifold with canonical class \(c\in H^2(M,\f{Z})\) and canonical bundle \(\lambda\).	Suppose \(c\) and the Pontryagin classes \(p_i(TM)\) are torsion and \(g_t\), \(t\in[0,1]\) is a smooth path of metrics on \(M\) with scal\((g_t)>0\).  If \(\nabla_0\) and \(\nabla_1\) are flat unitary connections on \(\lambda\), then 
	\[\eta(D^c_{g_0,\nabla_0})=\eta(D^c_{g_1,\nabla_1}).\]  
	
\end{thm}

\begin{proof}
	Modifying \(g_t\) if necessary we assume it is a constant path for \(t\) near \(0\) and \(1\).  Given \(L\in\ar{}_{>0}\), define a smooth metric \( g\) on \(M\times[0,1]\) by
	\[{g}=g_t+L^2dt^2.\]
 Then \( g\) is product-like near \(M\times\{0,1\}\).  One sees that scal\(({g})\) differs from scal\((g_t)\) by terms depending on the second fundamental form of each slice \(M\times{\{t\}}\), but the second fundamental form tends to 0 as \(L\to\infty\), so for large \(L\) we have scal\(( g)>0\).
	
	The difference of unitary connections on a complex line bundle is an imaginary one form.  Define \(\alpha\in\Omega(M)\) such that
	\[i\alpha=\nabla_1-\nabla_0.\]  Since both connections are flat, \(d\alpha=0\).  Let \(\pi:M\times [0,1]\to M\) be the projection and let \(f:M\times[0,1]\to[0,1]\) be the projection onto \([0,1]\) followed by be a smooth function which is \(0\) in a neighborhood of \(0\) and \(1\) in a neighborhood of \(1\). Define a connection  on \(\pi^*\lambda\) by \[{\nabla}=\pi^*\nabla_0+if\pi^*\alpha.\] Then, since \(\nabla_0\) is flat,
	\[F^\nabla=idf\wedge\pi^*\alpha.\]
	Let \({e_i}\) be an orthonormal frame for \({g}\) at a point \((p,t)\), such that \(e_1=\ov{L}\partial_t\).  Then
	\[2\sum_{i<j}|(df\wedge\alpha)(e_i,e_j)|=\frac{2\partial_tf}{L}\sum_{i>1}\alpha(e_i).\]  Since  \(e_i\), \(i>2\), is tangent to \(M\times\{t\}\), it does not  depend on \(L\).  Using \eqref{norm},  for large \(L\) we have
	\begin{equation*}\label{prodcurv}\text{scal}({g})>2|F^{\nabla}|_{{g}}. \end{equation*}
	The definition of \(f\) ensures that \({\nabla}\) is product-like near \(\partial(M\times I)\).  Then by \eqref{boch}  \(D^c_{g,\nabla}\) has trivial kernel and ind\((D^c_{g,\nabla}|_{S^+})=0.\)   
	
	Since \(F^{\nabla_i}=0\) for \(i=1,2\)
	\[\text{scal}(g_i)>0=2|F^{\nabla_i}|_{g_i}\] and hence \eqref{boch} implies  ker\(D^c_{g_i,\nabla_i}=\{0\}.\)    We now apply the Atiyah-Patodi-Singer index theorem \eqref{aps}.  The boundary of \(M\times I\) is two copies of \(M\) with opposite orientations.  The spectrum of the Dirac operator on \(M\times\{0,1\}\) is the union of the spectra on  \(M\times\{0\}\) and \(M\times\{1\}\), and the \(\eta\) invariant is the sum of the two \(\eta\) invariants.  When we change the orientation of an odd dimensional manifold, the Dirac operator changes by a sign.  Thus the Atiyah-Patodi-Singer theorem yields

	\[\text{ind}(D^c_{g,\nabla}|_{S^+})=\int_{M\times[0,1]}e^{c_1(\nabla)/2}\hat{A}(p(g))\]\[-\ov{2}\left({\text{dim}(\text{ker}(D^c_{g_0,\nabla_0}))+\text{dim}(\text{ker}(D^c_{g_1,\nabla_1}))+\eta(D^c_{g_0,\nabla_0})-\eta(D^c_{g_1,\nabla_1})}\right)\]
	
	and hence
	
	\[\eta(D^c_{g_1,\nabla_1})-\eta(D^c_{g_0,\nabla_0})=2\int_{M\times[0,1]}e^{c_1({\nabla})}\hat{A}(p({g})).\]
	
	\noindent Since \(\pi_1^*c\) is torsion, \(c_1(\nabla)\) is exact. Because  \(\nabla\) is flat near the boundary  \(c_1(\nabla)|_{\partial(M\times I)}=0\). Furthermore \({g}\) is product-like near the boundary so \(p(g)|_{M\times \{i\}}=p(g_i)\).  Since the real Pontryagin classes of \(M\) vanish \(p_j({g_i})\) is exact for \(j>0\). By Stokes' theorem, and since the dimension of \(M\) is \(4n+1\), the integral vanishes.        
\end{proof}

As a corollary we show how to use the \(\eta\) invariant to detect path components of moduli spaces of metrics with curvature conditions no weaker than positive scalar curvature. 

\begin{cor}\label{moduli}
	Let \(M\) be as in \tref{etta}.  Let \((g_i,\nabla_i)\) be a sequence of Riemannian metrics \(g_i\) with Ric\((g_i)>0,\) and flat  connections \(\nabla_i\) on \(\lambda\) such that \(\{\eta(D^c_{g_i,\nabla_i})\}_i\) is infinite.  Then \(\mathfrak{M}_{\text{\normalfont Ric}>0}(M)\) and \(\mods{\normalfont scal}(M)\) have infinitely many path components.       
\end{cor}  

\begin{proof}
	Let Diff\(^c(M)\) be the set of diffeomorphisms of \(M\) which fix the spin\(^c\) structure.  For \(g\in\mathfrak{R}_{\text{scal}>0}\) let \([g]\) represent the image in \(\mathfrak{M}_{\text{scal}>0}\) and \([g]^c\)  the image in \(\mathfrak{R}_{\text{scal}>0}/\text{Diff}^c(M)\).  It follows from Ebin's slice theorem (\cite{E}, \cite{Bo}),  that if \([g_i],[g_j]\) are in the same connected component of \(\mathfrak{R}_{\text{scal}>0}/\text{Diff}^c(M)\) then \(g_i,\phi^*g_j\) are in the same path component of \(\mathfrak{R}_{\text{scal}>0}\) for some \(\phi\in\text{Diff}^c(M)\).  Then there is a path between them maintaining positive scalar curvature, and by \tref{etta} and the spin\(^c\) diffeomorphism invariance of \(\eta\) we have  \(\eta(D^c_{g_i,\nabla_i})=\eta(D^c_{\phi^*g_j,\phi^*\nabla_j})=\eta(D^c_{g_j,\nabla_j}).\)  Since \(\{\eta(D^c_{g_i,\nabla_i})\}\) is infinite,  \(\mathfrak{R}_{\text{scal}>0}/\text{Diff}^c(M)\) has infinitely many components.  
	
	Any diffeomorphism \(\phi\) pulls back the spin\(^c\) structure to another one with canonical class \(\phi^*c\), a torsion class in \(H^2(M,\f{Z})\).  There are finitely many such classes. The finite group \(H^1(M,\f{Z}_2)\) indexes the spin\(^c\) structures associated to each class.  Thus the orbit of the spin\(^c\) structure under Diff\((M)\) and the set Diff\((M)/\)Diff\(^c(M)\) are finite.  The fibers  of \(\mathfrak{R}_{\text{scal}>0}/\text{Diff}^c(M)\to\mathfrak{M}_{\text{scal}>0}\) are no larger than Diff\((M)/\)Diff\(^c(M)\), implying that \(\mathfrak{M}_{\text{scal}>0}\) has infinitely many components.   
	
	The proof is identical for \(\mods{Ric}\) since Ric \( >0\) implies  scal \( >0\).     
\end{proof}

\section{\(\eta\) invariant in dimension \(4n+1\) with free  \(S^1\) actions}\label{proof}
In this section we prove \tref{thm}.  We want to use the Atiyah-Patodi-Singer index theorem to calculate the \(\eta\) invariant of a metric on \(M\).  Many authors have computed \(\eta\) and related invariants on spin manifolds \(M\) by extending metrics  to  manifolds \(W\) with boundary  diffeomorphic to \(M\).  If the extension has positive scalar curvature, the index of the Dirac operator will vanish.  In the spin\(^c\) case, we must also extend an auxiliary connection.  A difficulty arises when the extended connection cannot be flat because the canonical class of the spin\(^c\) structure on \(W\) is not torsion.  Then the metric and connection must satisfy \eqref{curvs}.  The following theorem, which we prove in \sref{mandc} illustrates how to use certain free \(S^1\) actions on \(M\) to accomplish this. 

\begin{thm}\label{mc}
	Let \(S^1\) act freely on \(M\) by isometries of a Riemannian metric \(g_M\) with scal\((g_M)>0\) and assume \(\pi_1(M)\) is finite.  Let \(B=M/S^1\) be the quotient and \(\rho:W=M\times_{S^1}D^2\to B\) the associated disc bundle.  Suppose the first Chern class of the principal \(S^1\) bundle \(\pi:M\to B\) is \(\ell d\) for \(d\in H^2(B,\f{Z})\) and \(\ell\in\f{Z}\).  If \(\lambda\) is the complex line bundle over \(W\) with first Chern class \(\rho^*d\), then there exists a metric \(g_W\) on \(W\) and a connection \(\nabla\) on \(\lambda\) such that 
	
	\begin{equation}\label{cineq}\text{\normalfont scal}(g_W)>{l}|F^\nabla|_{g_W}.\end{equation}
		Furthermore there is a collar neighborhood \(V\cong M\times [0,N]\) of \(\partial W\cong M\) such that for \(t\in[0,N]\) near \(0\), \(g_W\) is a product metric
	\begin{equation}\label{prodm}g_W\cong g_M+dt^2\end{equation}
	and 
	\begin{equation}\label{prodc}\nabla \cong \text{proj}_{V,M}^*\bar{\nabla}\end{equation}  where  \(\bar{\nabla}\) is any flat unitary connection on \(\lambda|_{\partial W}\).  \end{thm}

\noindent Notice that here there are no restrictions on the dimension or Pontryagin classes of \(M\),   \(d\) need not be primitive, and no spin\(^c\) structure is required.    We next use \tref{mc} and \eqref{aps} to calculate \(\eta\)  for \(S^1\) invariant metrics on certain spin\(^c\) manifolds in  dimensions \(4n+1\).  

\begin{thm}\label{etaeqn}
	Let \(S^1\) act freely on a \(4n+1\) manifold \(M\) by isometries of a Riemannian metric \(g\) with scal\((g)>0\).  Assume \(\pi_1(M)\) is finite and let \(B=M/S^1\) be the quotient.   Suppose the first Chern class of the principal bundle \(S^1\to M\xrightarrow{\pi} B\) is \(\ell d\) where \(\ell\) is a positive even integer and \(w_2(TB)=d\text{ \normalfont mod }2\).  Finally assume the real Pontryagin classes of \(M\) vanish. Then  \(M\) admits a spin\(^c\) structure with canonical class \(\pi^*d.\)  If  \(\bar\nabla\) is a flat connection on the canonical bundle of this spin\(^c\) structure and \(D^c_{g,\bar\nabla}\) is the spin\(^c\) Dirac operator,  then 
	
	\[\label{etaeqneqn}\eta(D^c_{g,{\bar\nabla}})=\left<\frac{\sinh(d/2)\hat{A}(TB)}{\sinh(\ell d/2)},[B]\right>.\]

\noindent   When \(n=1\), 
\begin{equation}\label{etaeqneqn4}
\eta(D^c_{g,\bar\nabla})=\left<-\frac{(\ell^2-1)d^2+p_1(TB)}{24\ell},[B]\right>.\end{equation}
\end{thm}
\begin{proof}
	Since \(TM\) is the direct sum of \(\pi^*TB\) and a trivial bundle generated by the action field of the \(S^1\) action,
	\[w_2(TM)=\pi^*w_2(TB)=\pi^*d\text{ mod }2\]  Let \(\mu\) be the complex line bundle over \(B\) associated to \(\pi:M\to B\).  Let \(W=M\times_{S^1}D^2\) and let \(\rho:W\to B\) be the disc bundle associated to \(\pi:M\to B\).  Then \(TW=\rho^*(TB\oplus\mu)\) and since \(\ell\) is even
	\[w_2(TW)=\rho^*(d+\ell d)\text{ mod }2=\rho^*d\text{ mod 2}.\] It follows that \(W\) admits a spin\(^c\) structure with canonical class \(\rho^*d.\)  We call the canonical bundle \(\lambda\). The spin\(^c\) structure on \(W\) induces one on \(M\) with canonical class \(\pi^*d\).  
	
	Then \(M,W,\) and \(\lambda\) satisfy the hypotheses of \tref{mc}.  We construct the metric \(g_W\) on \(W\) and connection \(\nabla\) on \(\lambda\) as in the theorem such that \(g_W|_M=g_M\) and \(\nabla|_M=\bar \nabla\).  Define the spin\(^c\) Dirac operator \(D^c_{g_W,\nabla}\) on \(W\) and \(D^c_{g_M,\bar\nabla}\) as in Section 2.  Given that \(g_W\) and \(\nabla\) are product-like near \(\partial W\), we can apply \eqref{aps}.  Since \(\bar\nabla\) is flat
	\[\text{scal}(g_M)>2|F^{\bar{\nabla}}|_{g_M}=0\]
	and by \eqref{cineq}
	\[\text{scal}(g_W)>\ell|F^{\nabla}|_{g_W}\geq 2|F^\nabla|_{g_W}.\]  
	Then \eqref{boch} implies that ind\((D^c_{g_W,\nabla})=0\) and ker\((D^c_{g_M,\bar{\nabla}})=\{0\}\).   It follows from \eqref{aps} that
	\begin{equation}\label{etaint}
	\eta(D^c_{g_M,\bar{\nabla}})=2\int_We^{c_1(\nabla)/2}\hat{A}(p(g_W)).\end{equation} 
	
	To evaluate that integral, we use Lemma 2.7 in \cite{KS1}:
	\begin{lem}\label{KSlem}\cite{KS}
		Let \(W\) be a manifold with boundary, and let \(\alpha,\beta\) be closed forms on \(W\) such that \(\alpha|_{\partial W}=d\hat\alpha\) and \(\beta|_{\partial W}=d\hat\beta.\)  Then 
		\[\int_W\alpha\wedge\beta=\int_{\partial W}\hat\alpha\wedge\beta+\left<j^{-1}(\alpha)\cup j^{-1}(\beta),[W,\partial W]\right>\]  
		where \(j^{-1}\) represents any preimage under the long exact sequence map 
		\[j:H^*(W,\partial W;\f{Q})\to H^*(W,\f{Q}).\]    
	\end{lem} 
	
\noindent	To apply \lref{KSlem} to \eqref{etaint}, let \(\alpha=e^{c_1(\nabla)/2}\) and \(\beta=\hat{A}(p(g_W)).\)  Since \(g_W\) is product like near the boundary, \(p_i(g_W)|_{\partial W}=p_i(g_M)\).  For \(i>0\) \(p_i(g_M)\) is exact by the assumption on the Pontryagin classes of \(M\).  Since   \(c_1(\nabla)|_{\partial W}=c_1(\bar\nabla)\) and \(\bar\nabla\) is flat, we can choose \(\hat\alpha=0\).  The form \(c_1(\nabla)\) represents the cohomology class \(c_1(\lambda)=\rho^*d.\)  Thus

	\[\eta(D^c_{g,\bar{\nabla}})=2\left<j^{-1}\left[e^{\rho^*d/2}\right]\cup j^{-1}\left[\hat{A}(TW)\right],[W,\partial W]\right>\]
	The following cup product diagram commutes:
	
	\[\begin{tikzcd}
	H^s(W,\partial W)\oplus H^t(W,\partial W)\arrow{r}{\cup}\arrow{d}{(\text{Id},j)}&H^{s+t}(W,\partial W)\arrow{d}\\H^{s}(W,\partial W)\oplus H^{t}(W)\arrow{r}{\cup}&H^{s+t}(W,\partial W)
	\end{tikzcd}\]
	Thus
	\begin{equation}\label{etacoho1}
	\eta(D^c_{g,\bar{\nabla}})=2\left<j^{-1}\left[e^{\rho^*d/2}\right]\cup \left[\hat{A}(TW)\right],[W,\partial W]\right>.
	\end{equation}
	
	Since the terms of \(\hat{A}(TW)\) have degree \(4k\), \(k\in\f{Z}\), and the dimension of \(W\) is \(4n+2\), only terms of degree \(4k+2\) in \(\displaystyle e^{\rho^*d/2}\) will contribute.  In those degrees, \(e^{\rho^*d/2}=\sinh(\rho^*d/2)\) as power series.  
	
	Since \(TW=\rho^*(TB\oplus \mu)\), \(\hat{A}(TW)=\rho^*(\hat{A}(TB)\hat{A}(\mu)).\)  For the complex line bundle \(\mu,\) we have 
	\[\hat{A}(\mu)=\frac{c_1(\mu)/2}{\sinh(c_1(\mu)/2)}=\frac{\ell d}{2\sinh (\ell d/2)}\]
	as a formal power series.  The series \(\sinh(d/2)\) is divisible by \(d,\) so 
	\[\rho^*\left(\frac{\sinh(d/2)}{\ell d}\right)\in H^*(W,\f{Q}).\]
	
	Let \(\Phi\in H^2(W,\partial W,\f{Z})\) be the Thom class of \(\rho: W\to B\).  Then \(j(\Phi)=\rho^*c_1(\mu)=\rho^*(\ell d).\) By means of another commutative diagram 
	\[\begin{tikzcd}
	H^*(W,\partial W)\oplus H^*(W)\arrow{r}{\cup}\arrow{d}{(j,\text{Id})}&H^*(W,\partial W)\arrow{d}{j}\\H^*(W)\oplus H^*(W)\arrow{r}{\cup}& H^*(W)
	\end{tikzcd}\]
	we see 
	\[j\left(\Phi\cup \rho^*\left(\frac{\sinh(d/2)}{\ell d}\right)\right)=\rho^*\left(\ell d \cup \frac{\sinh(d/2)}{\ell d} \right)=\rho^*\sinh(d/2).\] 
	Substituting into \eqref{etacoho1}
	\[\eta(D^c_{g,\bar{\nabla}})=2\left<\Phi\cup \rho^*\left(\frac{\sinh(d/2)}{\ell d}\right)\cup\rho^*\left(\frac{\hat{A}(TB)\cdot \ell d}{2\sinh(\ell d/2)}\right),[W,\partial W]\right>.\]
\[=\left<\Phi\cup \rho^*\left(\frac{\sinh(d/2)\hat{A}(TB)}{\sinh(\ell d/2)}\right),[W,\partial W]\right>.\]

	The Thom isomorphism yields
	\[\eta(D^c_{g,\bar{\nabla}})=\left<\frac{\sinh(d/2)\hat{A}(TB)}{\sinh(ld/2)},[B]\right>.\]
  
  When \(n=1\) the dimension of \(B\) is four and we have, as  series in \(H^*(B,\f{Z})\) 
  \[\hat{A}(TB)=1-\frac{p_1(TB)}{24}\]
  \[\frac{\sinh(d/2)}{\sinh(ld/2)}=\ov{\ell}\left(1-\frac{(\ell^2-1)d^2}{24}\right).\]
  
  Multiplying and isolating terms of degree four yields \eqref{etaeqneqn4}.
\end{proof}

\bigskip 

We now ready to prove \tref{thm}.  We first construct metrics of Ric\(>0\) on \(\cpcs{a}{b}\).  Perelman \cite{P} constructed a metric with Ric\(>0\) on arbitrary connected sums of \(\f{C}P^2\) with its standard orientation.  More details on Perelman's proof can be found in \cite{Bu} and \cite {BWW}.  With a slight adjustment to the construction one can reverse the orientation on some of the copies of \(\f{C}P^2\), proving the following:


\begin{lem}\label{per}
	\(\#^a\f{C}P^2\#^b\overline{\f{C}P^2}\) admits a metric with positive Ricci curvature for all \(a,b.\)
\end{lem}

\begin{proof}
	In \cite{P}, Perelman puts a metric on \(\#^c\f{C}P^2\) for all values of \(c\).  The construction involves \(c\) copies of \(\f{C}P^2\) attached to a central \(S^4\) by ``necks" \(S^3\times I.\)  The metric on the necks is of the form 
	\[ds^2=dt^2+A^2(t,x)dx^2+B^2(t,x)d\sigma^2\]
	where \(t\) is the coordinate on the interval \(I\) (see page 159 of \cite{P}).   Furthermore, \(S^3\) is represented as  the product of \(S^2\) and an interval with the top and bottom each identified to a point, and \(x\) is the coordinate on that interval, while \(d\sigma^2\) is the standard metric on \(S^2.\)  
	
	An orientation reversing isometry of \(d\sigma^2\), such as the antipodal map, extends naturally to a diffeomorphism of \(\phi:S^3\to S^3\) which induces an isometry of \(ds^2.\)  Let \(c=a+b,\) and take Perelman's metric on \(\#^c\f{C}P^2.\)  For \(b\) of the necks, we cut along a copy of  \(S^3\) and re-glue with \(\phi\) rather than the identity.  Because \(\phi\) reverses orientation, the resulting manifold is \(\#^a\f{C}P^2\#^b\overline{\f{C}P^2}\). Because \(\phi\) induces an isometry on \(S^3\times I,\) the same metrics on the pieces extend smoothly over the gluing, completing the proof of the lemma.  
\end{proof}

Let \(M^5\) satisfy the hypotheses of \tref{thm}.  By \lref{buntop}, \(M\) is Type III and by \tref{bases}  \(M\) is the total space of infinitely many non-isomorphic principal \(S^1\) bundles \(\pi_k:M^5\to B^4=\cpcs{a}{b}.\) From the proof of \tref{bases} we see that the first Chern class of \(\pi_k\) is \(2d_k\), where 
\[d_k=(1+2k,1,...,1)\in\zco{\cpcs{a}{b}}{2}\cong\f{Z}^{a+b}\] for a certain infinite set of integers \(k.\)

 Using the result of \cite{GPT} (see \lref{onemetric}) we see that since \(B\) admits a metric of positive Ricci curvature by \lref{per},  \(\pi_k:M\to B\) is a principal \(S^1\) bundle, and \(\pi_1(M)\) is finite, then for each \(k\) \(M\) admits a metric \(g_k\) with Ric\((g_k)>0\) such that the \(S^1\) action corresponding to the principal bundle \(\pi_k:M\to B\) acts by isometries of \(g_k\).    

Using the Gysin sequence it follows that \(H^4(M,\ar{})=0\) and \(M,g_k,\) and \(B\) satisfy the hypotheses of \tref{etaeqn} with \(g_M=g_k\), \(d=d_k\), \(\ell=2\) and \(\bar\nabla\) any flat connection on the canonical bundle of the spin\(^c\)  structure.  By  \eqref{etaeqneqn4}  we have

\[\eta(D^c_{g_k,\bar\nabla})=-\ov{16}\left(\left<c_k^2,[B]\right>+\text{sign}(B)\right)\]
\[=-\ov{16}\left(\pm 4k^2\pm4k+2\text{sign}(B)\right)\]
using the fact that \(\left<p_1(TB)/3,[B]\right>=\left<L(TB),[B]\right>\) is equal to the signature of \(B.\)  

Thus \(\eta(D^c_{g_k,\bar\nabla})\) is a nontrivial polynomial in \(k\) and takes on infinitely many values for the infinite set of integers \(k\).   \tref{moduli}  implies that \(\mods{Ric}(M)\) has infinitely many components, completing the proof of \tref{thm}.   
\qed

\

Note that \tref{moduli}  also implies that  \(\mods{scal}(M)\) has infinitely many components.

\section{metric and connection}\label{mandc}

In this section we prove \tref{mc}.  We first set up notation for the tangent space to \(W\).  We consider \(D^2\) to be the unit disc in \(\f{C}\).  Let \(\sigma:M\times D^2\to W\) be the quotient map so \(\sigma(p,x)=[p,x]\).  Then \(\rho([p,x])=\pi(p)\).  The metric \(g_M\) and the \(S^1\) action induce an orthogonal splitting \(T_pM=\bar{H}_p\oplus \bar{V}_p\) into horizontal space \(\bar H_p\) and vertical space \(\bar V_p\) of \(\pi\).  Define horizontal and vertical spaces of \(\rho\) to be

\[H_{[p,x]}=\sigma_*(\bar{H}_p\oplus\{0\})\] 
and
\[V_{[p,x]}=\sigma_*(\{0\}\oplus T_xD^2) \]
 
 for \(p\in M\) and \(x\in D^2.\)
 
 \smallskip
 
 These is well defined since for \(z\in S^1\), \(\bar{H}_{zp}=z_*\bar{H}_p\) and \(T_{zx}D^2=z_*T_{x}D^2.\)  One can use a local section of \(\sigma\) to see that \(H_{[p,x]}\) and \(V_{[p,x]}\)  are smooth distributions on \(W\).  Note that \(V_{[p,x]}\) is the tangent space to the fiber \(\rho^{-1}(\pi(p))=\sigma(\{p\}\times D^2)\) and \(T_{[p,x]}W=H_{[p,x]}\oplus V_{[p,x]}\).  Away from the zero section of \(\rho\), \(V_{[p,x]}\) is spanned by \[W_r=\sigma_*(0,\partial_r)\text{ and }W_\theta=\sigma_*(0,\partial_\theta).\]  These are well defined smooth vector fields since \(\partial_\theta\), \(\partial_r\) are \(S^1\) invariant vector fields on \(D^2\).         

Fix \(0<L<1\) and define a diffeomorphism 

\[\tau:M\times[L,1]\hookrightarrow M\times D^2\xrightarrow{\sigma}W\]

\smallskip

\noindent of \(M\times[L,1]\) to a collar neighborhood \(U\) of 
\(\partial W\).  Let \(t\) be the coordinate on \([L,1]\) and, in a slight abuse of notation, let \(\text{proj}_{U,M}:M\times[L,1]\to M\) be the projection.  Thus

\[\rho\circ\tau=\pi\circ\text{proj}_{U,M}\]
\[\tau_*(\bar{H}_p\oplus\{0\})=H_{[p,x]}\]
 \[\tau_*(0,\partial_t)=W_r \]
  
   Let \(X^*(p)=\left.\der{}{t}\right|_{t=0}e^{it}\cdot p\) be the action field of the \(S^1\) action on \(M\), which spams \(\bar V_p.\)  Then, since
   \(\sigma_*(X^*,\partial_\theta)=0,\) \[\tau_*(X^*,0)=-W_\theta.\]  Furthermore \(\tau|_{M\times\{1\}}\) identifies \(M\) and \(\partial W\), sending \(\bar{H}_p\) to \(H_{[p,1]}\) and \(X^*\) to \(-W_\theta\).

We keep track of the maps in the following diagram. 
\[\begin{tikzcd}
&M\times D^2 \arrow{d}{\sigma}
\\ M\times I \arrow{r}{\tau} \arrow[swap]{d}{\text{proj}_{U,M}} \arrow[hookrightarrow]{ur}& W \arrow{d}{\rho} \\
M \arrow{r}{\pi} &B
\end{tikzcd}
\]

\smallskip

To construct \(g_W\) and \(\nabla\) we will use two smooth functions on the interval \([0,1]\).   Let \(f_1:[0,1]\to[0,1]\) be a smooth monotone function which is \(0\) in a neighborhood of \(0\) and \(1\) in a neighborhood of \([L,1]\).

For a constant \(\ep>0\), let 
\[f_2(r)=-\ov{2}\int_0^rf_1(t)dt-\ep r^3+r.\]

One easily sees that \(f_2>0\) on \((0,1]\) for small \(\ep\).

\subsection{Metric}\label{metric}
We define a Riemannian metric at a point \((p,(r,\theta))\in M\times D^2\), where \(r,\theta\) are polar coordinates on \(D^2\), by
\[g_{M\times D^2}(p,(r,\theta))=g_M(p)+\ep^2|X^*(p)|_{g_M}^2\left(dr^2+\frac{f_2(r)^2}{1-\ep^2f_2(r)^2}d\theta^2\right).\]

\noindent By converting to Cartesian coordinates on \(D^2\)  , one sees that \(g_{M\times D^2}\) is smooth as long as 
\[\ov{r^4}\left(\frac{f_2^2}{1-\ep^2f_2^2}-r^2\right)\]
is a smooth function of \(r\in[0,1].\)  This is easily seen to hold since for \(r\) near 0, \(f_2(r)=r-\ep r^3.\) Since \(g_{M\times D^2}\) is invariant under the diagonal action of \(S^1\) on \(M\times D^2\), it induces a metric \(g_W\) on \(W\) such that \(g_{M\times D^2}\) and \(g_W\) make \(\sigma\)  into a Riemannian submersion.  Similarly, let \(g_B\) be the metric on \(B\) such that \(g_M\) and \(g_B\) make \(\pi\) into a Riemannian submersion.  

\begin{lem}\label{sub}
	\(g_W\) and \(g_B\) make \(\rho\) into a Riemannian submersion. 
	\end{lem}
\begin{proof}
With respect to \(g_{M\times D^2},\) \(\bar H_p \oplus\{0\}\) is orthogonal to \(X^*\) and \(TD^2.\)  Thus \(\bar H_p \oplus\{0\}\) is orthogonal to the vertical space of \(\sigma\), which is spanned by \((X^*,\partial_\theta)\), and to the horizontal projection of \(TD^2\) as well.  It follows that with respect to \(g_W,\) \(H_{[p,x]}\) is orthogonal to \(V_{[p,x]}\) and is the horizontal space of \(\rho.\)  Finally, we have
\[g_W|_{H_{[p,x]}}\cong g_{M\times D^2}|_{\bar H_p\oplus\{0\}}\cong g_M|_{\bar H_p}\cong g_B|_{T_{\pi(p)}B}.\]    
\end{proof}

We first describe the induced metric on the \(D^2\) fibers of \(\rho\). 

\begin{lem}\label{fibers}
	\[g_W|_{\rho^{-1}(\pi(p))}\cong\ep^2|X^*(p)|_{g_M}\left(dr^2+f_2(r)^2d\theta^2\right)\]
	\end{lem}

\begin{proof}
	\(\sigma|_{\{p\}\times D^2}:D^2\to\rho^{-1}(\pi(p))\) is a diffeomorphism such that \(\partial_r,\partial_\theta\) are mapped to \(W_r,W_\theta\).  Since \(\sigma\) is a Riemannian submersion with vertical space generated by \((X^*,\partial_\theta),\) we calculate
	
	\[|W_r|_{g_W}^2=|(0,\partial_r)|^2_{g_{M\times D^2}}=\ep^2|X^*|_{g_M}^2\]   
\[|W_\theta|_{g_W}^2=|(0,\partial_\theta)|^2_{g_{M\times D^2}}-\frac{\left<(0,\partial_\theta),(X^*,\partial_\theta)\right>^2_{g_{M\times D^2}}}{\left<(X^*,\partial_\theta),(X^*,\partial_\theta)\right>_{g_{M\times D^2}}}\]
\[=\ep^2|X^*|_{g_M}^2\left(\frac{f_2(r)^2}{1-\ep^2f_2(r)^2}\right)-\ep^4|X^*|_{g_M}^4\left(\frac{f_2(r)^2}{1-\ep^2f_2(r)^2}\right)^2\left(\ov{|X^*|_{g_M}^2+\ep^2|X^*|_{g_M}^2\left(\frac{f_2(r)^2}{1-\ep^2f_2(r)^2}\right)}\right)\]
\[=\ep^2|X^*|_{g_M}^2f_2(r)^2\]
\[\left<W_r,W_\theta\right>_{g_W}=\left<(0,\partial_r),(0,\partial_\theta)\right>_{g_{M\times D^2}}=0.\]
\end{proof}

We next modify \(g_W\) to have the desired product structure near \(\partial W.\)  We use a technique of Wraith, which allows deformations of metrics with positive mean curvature at the boundary.  

\begin{lem}\label{mean}
\(\partial W\) has positive mean curvature with respect to an inward normal vector.
\end{lem} 
\begin{proof}
Let \(\bar X_i\) be local \(S^1\)-invariant vector fields extending an orthonormal frame of \(\bar H_p\) and define \( X_i=\sigma_*(\bar X_i,0).\)  At a point \([p,1]\), \( \{X_i,\ov{\ep|X^*|_{g_M}f_2}W_\theta\}\) is an orthonormal basis of \(T\partial W\) and \(-\ov{\ep|X^*|_{g_M}}W_r\) is an inward pointing unit normal vector.   Since 
\[[ X_i,W_r]=[\sigma_*(\bar X_i,0),\sigma_*(0,\partial_r)]=\sigma_*[(\bar X_i,0),(0,\partial_r)]=0\]
and \(|X_i|=1,\)
\[\ov{\ep|X^*|_{g_M}}\left<\nabla_{ X_i} X_i,-W_r\right>=\ov{\ep|X^*|_{g_M}}\left< X_i,\nabla_{ X_i}W_r\right>=\ov{\ep|X^*|_{g_M}}\left< X_i,\nabla_{W_r} X_i\right>=0.\]
Thus  
\[\ov{\ep^3|X^*|_{g_M}^3f_2(1)^2}\left<\nabla_{W_\theta}W_\theta,-W_r\right>=\ov{2\ep^3|X^*|_{g_M}^3f_2(1)^2}W_r(|W_\theta|^2)=\frac{f'_2(1)}{\ep|X^*|_{g_M}f_2(1)}.\]
  Evaluating that quantity at \(r=1\) we see the the mean curvature is
  \[\frac{1/2-3\ep}{\ep|X^*|_{g_M}f_2(1)}>0\]
for sufficiently small \(\ep.\)
\end{proof}

We see that \(g_W|_{\partial W}\) is obtained from \(g_M\) by shrinking the \(S^1\) fibers of \(\pi\), a process which preserves positive scalar curvature. 

\begin{lem}\label{unshrink}
There exists a smooth path of metrics \(g_M(s)\) on \(M\), \(s\in[\ep^2f_2(1)^2,1]\), such that \(g_M(\ep^2f_2(1)^2)=g_W|_{\partial W},\) \(g_M(1)=g_M,\) and 
\(\text{scal}(g_M(s))>0\) for all \(s.\)  
\end{lem}

\begin{proof}
We recall that \(\tau|_{M\times\{1\}}:M\to\partial W\) is a diffeomorphism.  We see that 
\[(\left(\tau|_{M\times\{1\}}\right)^*g_W)|_{\bar H_p}=g_W|_{H_{[p,1]}}=g_M|_{\bar H_p}\]
and
\[|X^*(p)|^2_{\left(\tau|_{M\times\{1\}}\right)^*g_W}=|W_\theta([p,1])|^2_{g_W}=\ep^2f_2(1)^2|X^*(p)|^2_{g_M}.\]  Thus defining
\[g_M(s)=g_M|_{\bar H_p}+sg_M|_{\bar V_p}\]
we have, for \(\ep\) small enough, that \(\ep^2f_2(1)^2<1\),  \(g_M(\ep^2f_2(1)^2)=\left(\tau|_{M\times\{1\}}\right)^*g_W,\) and \(g_M(1)=g_M.\)  Since the metric is not changing on the horizontal space of \(\pi,\) each \(g_M(s)\) makes \(\pi\) into a Riemannian submersion with \(g_B.\)  The O'Neil formula \cite{Besse} then implies
\[\text{scal}(g_M(s))=\text{scal}(g_B)-s|A_\pi|^2-|T_\pi|^2-|N_\pi|^2-2\delta N_\pi\geq \text{scal}(g_M)>0\]
where \(A_\pi,T_\pi,N_\pi\) are the tensors defined for the Riemannian submersion \(\pi\) with respect to \(g_M.\)
\end{proof}

Use the normal exponential map from \(\partial W\) to define a collar neighborhood \(V\cong M\times [0,N],\) where \(t\in[0,N]\) is the distance to \(\partial W\).  We choose \(N\) small such that \(V\subset U\).  Using this identification, \(g_W\) has the form 
\[g_W=g(t)+dt^2\]
where \(g(t)={g_W}|_{M\times{t}}\) is a smooth path of metrics on \(M\).  Since \(g(0)={g_W}|_{\partial W}\) has positive scalar curvature, we can choose \(N\) small such that \(\text{scal}(g(t))>0\) for all \(t\in[0,N].\)

\begin{lem}\label{def}
	We can alter \(g_W\) inside of \(V\) such that it is product like near \(\partial W\) with \(g_W|_{\partial W}=g_M \) and scal\((g_W|_{V})>0\)
\end{lem}
\begin{proof}

We use the paths \(g_M(s)\) and \(g(s)\) and the following lemma from \cite{Wmo} to replace \(g_W\) near the boundary with a product metric restricting  to \(g_M\) at the boundary.

\begin{lem}{\cite{Wmo}}\label{W}
	Let \(g(t)+dt^2\) be a metric of positive scalar curvature on \(M\times[0,N]\) such that scal\((g(t))>0\) and \(M\times\{0\}\) has positive mean curvature with respect to the inward normal vector \(\partial_t\).  Let \(\bar g(t)\) be a smooth path of metrics on \(M\) such that \(\text{scal}(\bar g(t))>0\) for \(t\in[0,N]\) and \(\bar g(t)=g(t)\) for \(t\) in a neighborhood of \(N\).  Then there exists a function \(\beta:[0,N]\to\ar{}_+\) such that \(\beta=1\) for \(t\) in a neighborhood of \(N,\) \(\beta=\beta(0)\) is constant for \(t\) in a neighborhood of \(0,\) and \(\bar g(t)+\beta(t)dt^2\) has positive scalar curvature.

		\end{lem}

To define our replacement path \(\bar g\), we define two smooth functions.  

\[\chi_1:[0,N/2]\to[\ep^2f_2(1)^2,1]\text{ such that } \chi_1(t)=1\text{ for } t\text{ near }0\text{ and } \chi_1(t)=\ep^2f_2(1)^2\text{ for }t \text{ near } N/2 \]
\[\chi_2:[N/2,N]\to[0,1]\text{ such that } \chi_2(t)=0\text{ for } t\text{ near }N/2\text{ and } \chi_2(t)=t\text{ for }t \text{ near } N. \]
We then define a smooth path of metrics

\[\bar g (t)=\left\{\begin{array}{cc} g_M\circ\chi_1(t) & t\in[0,N/2]\\ g\circ \chi_2(t) & t\in[N/2,N]\end{array}\right..\]

      \noindent By \lref{unshrink} and the definition of \(g\),  \(\text{scal}(\bar g(t))>0\) for all \(t.\)  Then \lref{mean} and \lref{W} imply that \(\bar g(t)+\beta(t)dt^2\) has positive scalar curvature for the function \(\beta(t)\) given by \lref{W}.  For \(t\) near \(N\), \(\bar g(t)=g(t)\) and \(\beta(t)=1\) , so \(\bar g(t)+\beta(t)dt^2=g_W\).  Thus replacing \(g_W|_V\) with this metric results in a new smooth metric, for which we reuse the notation \(g_W.\)   Since \(\bar g(t)=g\) and \(\beta(t)\) is constant for \(t\) near \(0,\)  \(\bar g(t)+\beta(t)dt^2\) has the desired product structure \eqref{prodm}. This proves \lref{def}.  

\end{proof}

\subsection{Connection}
Let \(\beta\in \Omega^2(B)\) represent the image of \(\ell d\) in \(H^2(B,\ar{}).\)  The Gysin sequence for an \(S^1\) bundle shows that \(\pi^*ld=0,\) so we can choose \(\alpha\in\Omega^1(M)\) such that \(\pi^*\beta=d\alpha.\)  Since \(\pi^*\beta\) is \(S^1\) invariant, we can choose \(\alpha\) to be \(S^1\) invariant by averaging.  

\begin{lem}\label{thom}
\(\alpha(X^*)=-\ov{2\pi}\)\end{lem}
\begin{proof}
Let \(\Phi\in\Omega^2(W)\) be a Thom form of the disc bundle \(\rho:W\to B.\)  Since 
\[\left[\Phi\right]\mapsto\rho^*ld\]
under the long exact sequence map \(H^2(W,\partial W)\to H^2(W),\)  We have 
\[\rho^*\beta-\Phi=d\bar{\alpha}\]  for some  \(\bar{\alpha}\in\Omega^1(W)\).  Since \(\Phi\) vanishes near \(\partial W,\) \[d\bar{\alpha}|_{M}=\rho^*\beta|_M=\pi^*\beta=d\alpha.\]  

Since \(\pi_1(M)\) is finite, \(\bar\alpha|_M-\alpha\) is exact.  By the defining property of the Thom form, for any point \(q\in B\), \(\int_{\rho^{-1}(q)}\Phi=1\).  We use Stokes' theorem to compute  
\[-1=\int_{\rho^{-1}(q)}\rho^*\beta-\Phi=\int_{\rho^{-1}(q)}d\bar\alpha=\int_{\pi^{-1}(q)}\bar\alpha=\int_{\pi^{-1}(q)}\alpha=2\pi\alpha(X^*).\]

\end{proof}

We next construct a form \(\gamma\in\Omega^1(W)\) extending \(2\pi\alpha/\ell.\)  We first define a form \(\bar{\gamma}\in\Omega(M\times D^2)\).

At \((p,x)\in M\times D^2\), \(x\neq0\), set
\[\bar{\gamma}|_{\bar{H}_p\times\{0\}}=\frac{2\pi}{\ell}\alpha_{\bar{H}_p} \quad \bar{\gamma}(X^*,0)=-\frac{f_1(r)}{\ell}\]
\[\bar{\gamma}(0,\partial_r)=0 \quad \bar{\gamma}(0,\partial_\theta)=\frac{f_1(r)}{\ell}.\]

\noindent where \(r\) is the radial coordinate on \(D^2\).  This form extends smoothly to the origin of \(D^2\) since \(f_1\) is zero in a neighborhood of \(r=0\).    Since \(r\), \(\bar{H}_p\oplus\{0\}\), \(\alpha\), \(\partial_r\), \(\partial_\theta\), \(X^*\) are all preserved by the \(S^1\) action, \(\bar{\gamma}\) is \(S^1\) invariant.  The vertical space of \(\sigma\) is generated by \((X^*,\partial_\theta)\), and so \(\bar{\gamma}\) vanishes on the vertical space.  It follows that there is a unique form \(\gamma\in\Omega(W)\) such that \(\sigma^*\gamma=\bar{\gamma}\).   

\begin{lem}\label{gcol}
\(	\tau^*\gamma=\frac{2\pi}{\ell}\text{\normalfont proj}_{U,M}^*\alpha\)
	\end{lem}

\begin{proof} Recall that \(f_1(r)=1\) for \(r\) in the image of \(\tau\) and note that \(\tau^*\gamma=(\sigma^*\gamma)|_{M\times[L,1]}=\bar\gamma|_{M\times[L,1]}.\)  Thus:	\[\tau^*\gamma|_{\bar{H}_p\oplus\{0\}}=\bar{\gamma}|_{\bar{H}_p\oplus\{0\}}=\frac{2\pi}{\ell}\alpha_{\bar{H}_p},	\quad\quad \tau^*\gamma(X^*,0)=\bar\gamma(X^*,0)=-\frac{f_1(r)}{\ell}=\frac{2\pi}{\ell}\alpha(X^*)\]
		\[\text{and}\quad \tau^*\gamma(0,\partial_t)=\bar\gamma(0,\partial_r)=0=\frac{2\pi}{\ell}\alpha(\text{proj}_{M*}(0,\partial_t)).\]
\end{proof}
  Let \(\lambda_B\) be the complex line bundle with \(c_1(\lambda_B)=d\).   Given a differential form in the de Rahm cohomology class of \(2\pi i\) times the first Chern class of a complex line bundle, there is a unitary connection on the line bundle whose curvature is that differential form.   Thus, since \(\beta\) represents \(\ell d\),  let \(\nabla_B\) be a unitary connection on \(\lambda_B\) with curvature
\[F^{\nabla_B}=\frac{2\pi i}{\ell}\beta.\]

We  now define a connection on \(\lambda\)

\[\nabla=\rho^*\nabla_B-i\gamma.\]

\begin{lem}\label{ccol}
\(\nabla\) is flat on \(U.\)
	\end{lem}

\begin{proof}
	We need to show that \(F^{\tau^*\nabla}=0\).  Using \lref{gcol}  it follows that
	\[\tau^*\nabla=\tau^*\rho^*\nabla_B-i\tau^*\gamma\]
		\[=\text{proj}_{U,M}^*\left(\pi^*\nabla_B-\frac{2\pi i}{\ell}\alpha\right)\]
		and hence the curvature of the term in the parentheses is 
		\[\frac{2\pi i}{\ell}\pi^*\beta-\frac{2\pi i}{\ell}d\alpha=0.\]
	
	\end{proof}

We finish the construction of \(\nabla\) by modifying it so that it is product like near \(\partial W\) and restricts to \(\bar\nabla\) at \(\partial W\).  Let \(\text{proj}_{V,M}^*:V\to M\) be the projection defined by the identification \(V\cong M\times[0,N]\) from \sref{metric}.  Note that while \(V\subset U,\) \(\text{proj}_{V,M}\) and \(\text{proj}_{U,M}\)  will not in general agree (the later was defined independently of \(h\) and the former using \(h\).)  Since \(V\subset U,\) \(\nabla\) is flat on \(V.\)  Since \(\text{proj}_{V,M}\) and the inclusion of \(\partial W\cong M\times\{0\} \) are homotopy inverses, \(\text{proj}_{M,V}^*(\lambda|_M)=\lambda|_V.\)  Thus \(\nabla|_V\) and \(\text{proj}_{V,M}^*\bar{\nabla}\) are both flat unitary connections on \(\lambda|_V\) and 
\[\text{proj}_{V,M}^*(\bar\nabla)-\nabla|_V=i\delta\]
for some closed form \(\delta\in\Omega^1(V).\) Since \(\pi_1(V)=\pi_1(M)\) is finite, \(\delta=df\) for a smooth function \(f\) on \(V\).  We modify \(f\) to a function \(\bar{f}\) which is equal to \(f\) near \(\partial W\cong M\times \{0\}\) and equal to \(0\) near \(M\times\{N\}.\)  We then replace \(\nabla\) with \(\nabla+id\bar f\) on \(V.\)  We see that \(\nabla\) is still smooth, flat on \(V\), and near \(\partial W\),  \(\nabla=\text{proj}_{V,M}^*\bar{\nabla},\) satisfying \eqref{prodc}   

\subsection{Curvature}
We complete the proof of \tref{mc} by showing that \eqref{cineq} holds.  On \(V,\) \(\nabla\) is flat and by \lref{def} \(\text{scal}(g_W)>0,\) so the inequality is satisfied.  For the remainder of the proof we consider \(W\backslash V\).   Then \(\text{scal}(g_W)\) is given by \lref{sub} and the O'Neil formula for the scalar curvature of Riemannian submersion 
\[\text{scal}(g_W)=\text{scal}\left(g_W|_{\rho^{-1}(\pi(p))}\right)+\text{scal}(g_B)-|A_\rho|^2-|T_\rho|^2-|N_\rho|^2-2\delta N_\rho.\]

As \(\ep\to0,\) \(|A_\rho|\to0,\) while the final three terms remain constant.  By \lref{fibers}, 
\[\text{scal}\left(g_W|_{\rho^{-1}(\pi(p))}\right)=-\frac{2}{\ep^2|X^*|^2_{g_M}}\left(\frac{f_2''}{f_2}\right).\]  

Therefore, as \(\ep\to 0,\)
\[\text{scal}(g_W)=-\frac{2}{\ep^2|X^*|^2_{g_M}}\left(\frac{f_2''}{f_2}\right)+O(1).\]

 Let \(\bar X_i\) be an orthonormal basis of \(\bar{H}_{p}\) with respect to \(g_M\). Let \(X_i=\sigma_*(\bar{X}_i,0)\).  Then \(\{X_i\}\) is an orthonormal basis of \(H_{[p,x]}\) with respect to \(g_W\) outside of \(V\).  Away from the zero section \(\{\ov{\ep|X^*|_{g_M}}W_r,\ov{\ep|X^*|_{g_M}f_2}W_\theta\}\) is an orthonormal basis of \(V_{[p,x]}\).   Neither the \(\bar X_i\) nor \(\nabla\) depend on \(\ep\).  Then as \(\ep\to0\), using \eqref{norm}
\[\left|F^{{\nabla}}\right|_{g_M}\leq\ov{\ep^2|X^*|^2_{g_M}f_2}|F^{{\nabla}}(W_r,W_\theta)|+\sum_i\ov{\ep|X^*|_{g_M}}|F^{{\nabla}}(W_r,X_i)|+\ov{\ep|X^*|_{g_M}f_2}|F^{{\nabla}}(W_\theta,X_i)|+O(1).\]     

\begin{lem}\label{vals}\(F^\nabla(W_r,W_\theta)=-if_1'(r)/\ell\) ,  \label{formv}\(F^\nabla(W_r,X_i)=F^\nabla(W_\theta,X_i)=0\).\end{lem}

\begin{proof}
	Since \(\rho_*W_r=\rho_*W_\theta=0,\)
\[F^\nabla(W_r,W_\theta)=-id\gamma(W_r,W_\theta)=-id\gamma(\sigma_*(0,\partial_r),\sigma_*(0,\partial_\theta))\]
\[=-i\sigma^*d\gamma((0,\partial_r),(0,\partial_\theta))=-id\bar\gamma((0,\partial_r),(0,\partial_\theta))=-i\partial_r\bar\gamma(0,\partial_\theta)=-i\frac {f_1'(r)}{\ell}\]
similarly 
\[\F^\nabla(W_r,X_i)=-id\bar\gamma((0,\partial_r),(\bar X_i,0))=-i\left(\partial_r\left(\frac{2\pi}{\ell}\alpha(\bar X_i)\right)-\bar X_i\left(\frac{f_1(r)}{\ell}\right)\right)=0\]
and 
\[\F^\nabla(W_\theta,X_i)=-id\bar\gamma((0,\partial_\theta),(\bar X_i,0))=-i\left(\partial_\theta\left(\frac{2\pi}{\ell}\alpha(\bar X_i)\right)\right)=0.\]
\end{proof}

\lref{vals} implies that as \(\ep\to0,\)
\[\text{scal}(g_W)-\ell|F^\nabla|_{g_M}=\ov{\ep^2|X^*|^2_{g_M}}\left(\frac{-2f_2''-f_1'}{f_2}\right)+O(1)\]
\[=\frac{12}{\ep|X^*|^2_{g_M}}\left(\frac{r}{f_2}\right)+O(1)\]

From the definition of \(f_2\) one sees that \(r/f_2\to 1\) as \(r\to0.\)  It follows that we can choose \(\ep\) small enough that \eqref{cineq} holds, completing the proof of \tref{mc}.  
\qed

\bigskip

  In \cite{KS} Lemma 4.2, Kreck and Stolz constructed positive scalar curvature metrics on  associated disc bundles in order to calculate their invariant for spin manifolds with free \(S^1\) actions.  In their proof, they needed to assume that the \(S^1\) orbits were geodesics.  The metric \(g_W\) constructed in \tref{mc} generalizes their method to a free isometric \(S^1\) action without the geodesic condition.

\end{document}